\documentclass[10pt,twocolumn]{article} 
\usepackage[a4paper, left=1.5cm, right=1.2cm, top=1.5cm, bottom=2.5cm]{geometry}

\usepackage[utf8x]{inputenc}
\usepackage{cite}
\usepackage{amsmath,amssymb,amsfonts}

\usepackage{amsthm}
\usepackage{algorithmic}
\usepackage{graphicx}
\usepackage{textcomp}
\usepackage{mathrsfs}
\usepackage{exscale}
\usepackage{booktabs}
\usepackage{float}
\usepackage{bm}
\usepackage{xcolor}
\usepackage{dsfont}
\usepackage{todonotes}
\usepackage{array}
\usepackage[normalem]{ulem}
\usepackage{siunitx}

\graphicspath{{fig/}}

\theoremstyle{theorem}
	\newtheorem{theorem}{Theorem}[section]
	\newtheorem{lemma}[theorem]{Lemma}
	\newtheorem{problem}[theorem]{Problem}
	\newtheorem{proposition}[theorem]{Proposition}
	\newtheorem{corollary}[theorem]{Corollary}
	
\theoremstyle{definition}

\theoremstyle{remark}
	\newtheorem{remark}[theorem]{Remark}
	\newtheorem{assumption}[theorem]{Assumption}
	\numberwithin{equation}{section}

\def\I{\mathrm{I}}

\def\L2{{\cal L}_2}
\def\L2e{{\cal L}_{2e}}

\def\R{\mathbb{R}}

\def\begequarr{\begin{eqnarray}}
\def\endequarr{\end{eqnarray}}
\def\begequarrs{\begin{eqnarray*}}
	\def\endequarrs{\end{eqnarray*}}
\def\begarr{\begin{array}}
	\def\endarr{\end{array}}
\def\begequ{\begin{equation}}
\def\endequ{\end{equation}}

\def\begdes{\begin{description}}
	\def\enddes{\end{description}}
\def\begenu{\begin{enumerate}}
	\def\begite{\begin{itemize}}
		\def\endite{\end{itemize}}
	\def\endenu{\end{enumerate}}

\def\lef[{\left[\begin{array}}
	\def\rig]{\end{array}\right]}

\def\begcen{\begin{center}}
	\def\endcen{\end{center}}
\def\begrem{\begin{remark}\rm}
	\def\endrem{\end{remark}}

\begin{document}
	\title{Data-Driven Control for Linear Discrete-Time Delay Systems}

\author{Juan G. Rueda-Escobedo, Emilia Fridman and Johannes Schiffer
	\thanks{J.G. Rueda-Escobedo and J. Schiffer are with the Control Systems and Network Control Technology Group, Brandenburg University of Technology Cottbus-Senftenberg (BTU C-S), 03046 Cottbus, Germany (e-mail: \{ruedaesc,schiffer\}@b-tu.de).}
	\thanks{E. Fridman is with the Department of Electrical Engineering-Systems, Tel-Aviv University, Ramat-Aviv, Tel-Aviv 69978, Israel (e-mail: emilia@eng.tau.ac.il)}
}
\maketitle

\begin{abstract}
	The increasing ease of obtaining and processing data together with the growth in system complexity has sparked the interest in moving from conventional model-based control design towards data-driven concepts.  
	Since in many engineering applications time delays naturally arise and are often a source of instability, we contribute to the data-driven control field by introducing data-based formulas for state feedback control design in linear discrete-time time-delay systems with uncertain delays. With the proposed approach, the problems of system stabilization as well as of guaranteed cost and $H_{\infty}$ control design are treated in a unified manner. Extensions to determine the system delays and to ensure robustness in the event of noisy data are also provided.
	
\end{abstract}


\section{Introduction}
	There is a growing stream of efforts for developing novel control design methods that only rely on data, enabling a direct control synthesis while avoiding intermediate steps, such as system modeling or system identification \cite{DataDriven2012,HouWang2013}. 
	This trend is driven by several factors. These comprise the increasing ease of obtaining and processing data, which is facilitated by modern computers and communication networks, the growth in system complexity in many modern applications and the desire of systematizing the control design. 
	
	Although this movement has its roots in computer science~\cite{HouWang2013}, where, among other techniques, neural networks, fuzzy systems, online optimization, learning methods, etc., are used for system control, the area of data-driven control has recently shifted towards the development of controller synthesis approaches, which are based on more conventional control strategies. 
	A main reason for this is the need of rigorous guarantees on the system operation, in other words, the need of robust controllers. With this premise in mind, there has been a number of recent contributions in the area of linear system control. The main idea is to assume an underlying linear system to interpret the data and to develop data-driven control formulas which leads to robust controllers by accounting for model mismatches, noise and disturbances. Recent contributions comprise works on linear quadratic tracking \cite{MarkRap2007,Markovsky08}, dynamical feedback \cite{ParkIkeda2009}, predictive control \cite{Coulson2019,Coulson2019_2,Berberich2020}, state-feedback and optimal control \cite{PangBian2018,Tu2018,Rotulo2019,Dean2019,Persis2019,Persis2019_2,vanWaarde2019,vanWaarden2020,Berberich2020_2} 
	as well as extensions to nonlinear discrete Volterra systems and subclasses thereof \cite{BerAll2019,Rueda2020}.
	
	Among the most relevant robust control problems is the stabilization of time-delay systems (TDSs). 
	Time delays are an ubiquitous phenomenon in many engineering applications, such as biological and chemical systems as well as networked control and sampled-data systems \cite{Richard2003, Fridman2014,LiuFridman2019}. Yet in this important direction, to date there are only few contributions from a data-driven control perspective. One of these is \cite{Formentin2011}, where the authors extend the Virtual Reference Feedback Tuning (VRFT) method to single-input single-output (SISO) linear discrete-time (LDT) TDS with known input delay. The VRFT is combined with a data-based Smith predictor to account for the effect of the delay. A similar approach is presented in \cite{Kaneko2011} for a SISO linear continuous-time TDSs with unknown input delay. In the field of optimal control for TDSs, a data-driven quadratic guaranteed cost control for continuous time TDSs with known delay, but unknown system matrices is proposed in \cite{Jiang2019}. Therein, the system data is used to characterize the cost and to update the control gains. In a similar direction, in \cite{Liu2018,Liu2019}, the authors propose a data-based adaptive dynamic programming method for optimal and $H_{\infty}$ control design.	
	
These recent advances motivate the work in the present paper, which is focused on data-driven control design for LDT-TDSs with state and input delays. Inspired by  \cite{Persis2019} and \cite{Fridman2014}, we provide data-driven formulas for the computation of state feedback gains to achieve system stabilization as well as for guaranteed cost and $H_{\infty}$ control design in a unified manner. In contrast to the approaches in \cite{Formentin2011,Liu2018,Liu2019}, the proposed method addresses the case of uncertain and time-varying delays. Furthermore, the impact of noise in the data is analyzed and taken into account for the feedback design, resulting in robust stability guarantees for the closed-loop system. More precisely, the following contributions are made:	
	\begin{enumerate}
		\item	From input-state data and for {\em known} delays, we provide data-based formulas to replace the system model by the data itself. These formulas can be used for system representation or for control design.
		\item	By using these data-based formulas, we provide data-driven formulas for the design of state-feedback gains. These formulas are given for three control problems: stabilization, control with guaranteed cost, and for $H_{\infty}$ control. In all these cases, uncertain delays are considered.
		\item The proposed approach is extended to the cases of {\em unknown} delays and data corrupted by noise.
		In particular, we provide an algorithm to determine the system delays from disturbed data, and we robustify the data-driven formulas to account for the impact of noise.
	\end{enumerate}

	The organization of the paper is as follows. In Section \ref{Sec:Preliminaries} the addressed system is described and the main goals of the paper are outlined. In Section \ref{Sec:DataRepresentation}, a data-based representation for linear discrete-time time-delay systems is introduced. By using the results of Section \ref{Sec:DataRepresentation}, in Section \ref{Sec:ControlFormulas} data-based formulas for control design are given. The formulas address three basic control problems: stabilization, control with guaranteed cost, and $H_{\infty}$ control. In Section \ref{Sec:Uncertainties} we investigate the effect of uncertainties in the data. The application of the control formulas is illustrated with a numerical example in Section \ref{Sec:Example}. In Section \ref{Sec:Conclusion} some concluding remarks are given. Finally, the proofs of all the claims are given in the Appendix.

	\subsection{Notation}
		The set of integer numbers is denoted by $\mathbb{Z}$ and $\mathbb{R}$ represents the set of real numbers. Let $\mathbb{F}$ be either $\mathbb{Z}$ or $\mathbb{R}$. Then $\mathbb{F}_{> 0}$ ($\mathbb{F}_{\geq 0}$) denotes the set of all elements of $\mathbb{F}$ greater than (or equal to) zero. The identity matrix of order $n\in\mathbb{Z}_{>0}$ is denoted by $I_n$. For $A\in\mathbb{R}^{n\times n},$ $A>0$ means that $A$ is symmetric positive definite. The elements below the diagonal of a symmetric matrix are denoted by $\star$. Given a matrix $A\in\mathbb{R}^{n\times m}$, $A^\dagger$ denotes its Moore-Penrose inverse. If $A$ has full-row rank, we have $AA^\dagger=I_n$ with
		\begin{align*}
			A^{\dagger}=A^\top\big(AA^\top\big)^{-1}.
		\end{align*}
		For $v\in\mathbb{R}^n$, $\|v\|_2=\sqrt{v^\top v}$ denotes the Euclidean norm of $v$. For $A\in\mathbb{R}^{m\times n}$, $\|A\|_2=\max_{\|v\|_2=1}\|A\,v\|_2$  with $v\in\mathbb{R}^n$ denotes the induced Euclidean norm of $A$.
				
		Given a signal $z:\mathbb{Z}\to\mathbb{R}^n$ and two integers $k$ and $r$, where $r\geq k$, we define $z_{[k,r]}:=\left\{z(k),z(k+1),\cdots,z(r)\right\}$. Given a signal $z$ and a positive integer $T$, we define
		\begin{align}
			\label{Eq:Notation}
			Z_{\{i\}}=Z_{\{i,T\}}:=\begin{bmatrix}z(i) & z(i+1) \cdots & z(T+i-1)\end{bmatrix}.
		\end{align}

		Finally, given signals $x(k)\in\mathbb{R}^n$ and $h(k)\in\mathbb{Z}_{\geq 0}$ for $k\in\mathbb{Z}$, we introduce the short hand $x_{h(k)}(k):=x(k-h(k))$. If $h$ is independent of $k$, then $x_h(k):=x(k-h)$.

\section{Considered Class of Systems and Objectives}
\label{Sec:Preliminaries}
	The following LDT-TDS is considered in this paper:
	\begin{align}
		x(k+1)&=A_0x(k)+A_1x_{h_{1}(k)}(k)+Bu_{h_{2}(k)}(k),
		\label{Eq:System}
	\end{align}
	with $k\in\mathbb{Z}_{\geq 0}$, state vector $x(k)\in\mathbb{R}^n$ and input $u(k)\in\mathbb{R}^m.$ Furthermore, $h_1(k)$ and $h_2(k)$, where $h_1(k)\in\mathbb{Z}_{\geq 0}$ and $h_2(k)\in\mathbb{Z}_{\geq 0}$, represent uncertain, bounded delays with upper bound $\bar{h}\geq h_i(k)$ for $i=\{1,2\}$ and all $k\in\mathbb{Z}_{\geq 0}$. 
	With respect to the system's initial condition and past inputs we assume $x(j)=\phi_j\in\mathbb{R}^n$ and $u(j)=0$ for $j\in\mathbb{Z}\cap[-\bar{h},0]$. 
	
	The main objective of this paper is to design state feedback controllers {\em directly from input-state data} that stabilize the system \eqref{Eq:System} in the presence of - possibly uncertain - delays $h_{1}(k)$ and $h_{2}(k)$. The analysis is conducted under the following assumptions on the system \eqref{Eq:System}.
	\begin{assumption}
		\phantom{M}
		\begin{enumerate}
			\item	The system matrices $A_0$, $A_1$ and $B$ are constant but unknown.
			\item	An upper bound $\bar{h}\in\mathbb{Z}_{> 0}$ for the input and state delay $h_{1}(k)$ and $h_{2}(k)$, respectively, is known.
			\item	Input and state sequences $u_{[-\bar{h},T]}$ and $x_{[-\bar{h},T]}$ are available, where $T\in\mathbb{Z}_{> 0}$ with $T> \bar{h}$ is the number of recorded samples and the delays $h_1$ and $h_2$ were constant during the time window in which the data was recorded.
		\end{enumerate}
		\label{Assum:Main}
	\end{assumption}

	Assumption \ref{Assum:Main}.1 and Assumption \ref{Assum:Main}.2 are standard. 
	Assumption \ref{Assum:Main}.3 can be contextualized as follows. Consider a scenario in which the recorded data is produced in a controlled experiment where the state and input delays are constant. However, during the system operation, the delays might change. Another scenario in which Assumption \ref{Assum:Main}.3 is reasonable is in networked control. Suppose that for generating the data the system is operated in open-loop and that the input delay remains constant. Then the system state can be recorded locally and transmitted later for its processing. Hence, Assumption \ref{Assum:Main}.3 follows. However, when the system is operated in closed-loop, the input delay becomes uncertain (but bounded) due to the transmission of the state measurement through the network \cite[Sec.~7.8.1]{Fridman2014}, \cite[Sec.~3]{LiuFridman2019}.

	In the following, a data-based representation framework is introduced for the system \eqref{Eq:System} under Assumption \ref{Assum:Main}. At first, this is done for the case of {\em known} delays. By using the resulting framework, three control designs are derived in Section~\ref{Sec:ControlFormulas}. These are stabilizing control, guaranteed cost control, and $H_\infty$~control. In addition, in Section \ref{Sec:Uncertainties} the proposed approach is extended to the case of {\em unknown} delays and noisy data.

\section{Data-Based System Representation of Linear Discrete-Time Time-Delay Systems}
\label{Sec:DataRepresentation}
	In this section, we derive a data-based representation of the system \eqref{Eq:System} using the data provided by the sequences $u_{[-\bar{h},T]}$ and $x_{[-\bar{h},T]}$, see Assumption \ref{Assum:Main}.3. To this end, we at first assume that the delays $h_1$ and $h_2$ are known and constant. The case of unknown delays is then treated in Section~\ref{Sec:Uncertainties}. 	Under these considerations, the system \eqref{Eq:System} can be rewritten as
	\begin{align}
		x(k+1)&=\begin{bmatrix}B & A_1 & A_0\end{bmatrix}\left[\begin{array}{c}u_{h_2}(k)\\ x_{h_1}(k)\\ x(k)\end{array}\right]. \label{Eq:System2}
	\end{align}
	To represent the system \eqref{Eq:System2} solely by data, consider also the matrices $U_{h_2,\{0\}}\in\mathbb{R}^{m\times T}$, $X_{h_1,\{0\}}\in\mathbb{R}^{n\times T}$, $X_{\{0\}}\in\mathbb{R}^{m\times T}$, $X_{\{1\}}\in\mathbb{R}^{m\times T}$, and
	\begin{gather}
			W_{0}:=\left[\begin{array}{c}U_{h_2,\{0\}}\\ X_{h_1,\{0\}}\\ X_{\{0\}}\end{array}\right]\in\mathbb{R}^{(m+2n)\times T} \label{Eq:W0}.
	\end{gather}
	Here, the matrices $U_{h_2,\{0\}}$, $X_{h_1,\{0\}}$, $X_{\{0\}}$ and $X_{\{1\}}$ are built using the sequences corresponding to $u_{h_2}(k)$, $x_{h_1}(k)$, $x(k)$ and $x(k+1)$, respectively, in accordance with the definition given in \eqref{Eq:Notation}. We have the following result.

	\begin{proposition}[Open-Loop Data-Based Representation] 
		\label{Prop:OpenLoopRep}		
		The system trajectories of \eqref{Eq:System2} and the ones of the system
		\begin{align}
			x(k+1)&=X_{\{1\}}W^{\dagger}_0\left[\begin{array}{c}u_{h_2}(k)\\ x_{h_1}(k)\\x(k) \end{array}\right] \label{Eq:OpenLoopRep}
		\end{align}
		are equivalent if and only if $W_0$ given in \eqref{Eq:W0} satisfies $\mathrm{rank}(W_{0})=m+2n$. Furthermore, it holds that $[B\;A_1\;A_0]=X_{\{1\}}W^{\dagger}_0$.
		\hfill $\triangledown\triangledown\triangledown$
	\end{proposition}	
	The condition on the rank of $W_0$ in Proposition~\ref{Prop:OpenLoopRep} is equivalent to the requirement that the recorded data is rich enough. Since this rank condition is necessary and sufficient, it is analogous to \cite[Eq.~6]{Persis2019}, but for LDT-TDSs of the form \eqref{Eq:System2}. The rank condition $\mathrm{rank}(W_0)=m+2n$ will appear repeatedly along this note. From this rank condition it also follows that a minimal requirement on the data length $\ell$ is that $\ell \geq m+2n+\bar h$ (since the sequences $u_{h_2}(k)$, $x_{h_1}(k)$, $x(k)$ and $x(k+1)$ are used to build $W_0$ in \eqref{Eq:W0}).
		
	In a similar way, one can find a system representation in closed-loop by using the recorded data. While Proposition \ref{Prop:OpenLoopRep} represents an identification-like result, the following lemma provides a system representation that can be used for control design while avoiding the identification of the system matrices.
	
	\begin{lemma}[Closed-Loop Data-Based Representation]
		\label{Lem:ClosedLoopRep}
		Consider the system \eqref{Eq:System2} and assume a feedback control of the form $u(k)=Kx(k)$ with $K\in\mathbb{R}^{m\times n}$. The trajectories of the closed-loop system
		\begin{align}
			x(k+1)&=\begin{bmatrix}BK & A_1 & A_0\end{bmatrix}\left[\begin{array}{c}x_{h_2}(k)\\ x_{h_1}(k)\\ x(k)\end{array}\right]
			\label{Eq:ClosedLoopSys}
		\end{align}
		and the ones of the system
		\begin{align}
			x(k+1)&=X_{\{1\}}G_{K}\left[\begin{array}{c}x_{h_2}(k)\\ x_{h_1}(k)\\ x(k)\end{array}\right] \label{Eq:DataRepCL},
		\end{align}
		where $G_K$ is a $T\times 3n$ matrix satisfying
		\begin{align}
			\begin{bmatrix}K & 0 & 0\\0 & I_n & 0\\ 0 & 0 & I_n\end{bmatrix}=W_0G_{K},
			\label{Eq:GK}  
		\end{align}
		are equivalent for every $K$ if and only if $\mathrm{rank}(W_0)=m+2n$ with $W_0$ given in \eqref{Eq:W0}. In particular, one has 
		\begin{align}
			u(k)=U_{h_2,\{0\}}G_K\begin{bmatrix}I_n & 0 & 0\end{bmatrix}^\top x(k).
			\label{Eq:Kformula}
		\end{align}
		\hfill $\triangledown\triangledown\triangledown$
	\end{lemma}
	Note that Lemma \ref{Lem:ClosedLoopRep} not only provides a purely data-based representation for the closed-loop system, but also for the control input, and more importantly, for the feedback gain. These characteristics are exploited in the next section for controller synthesis.

	\begin{remark}
		By setting $h_1$ or $h_2$ in \eqref{Eq:System2} to zero, respectively, the following standard LDT-TDS can be described using the same approach as in Proposition~\ref{Prop:OpenLoopRep}:
		\begin{itemize}
			\item 	$x(k+1)=A_0x(k)+Bu(k-h_2)$,
			\item	$x(k+1)=A_0x(k)+A_1x(k-h_1)+Bu(k)$.
		\end{itemize}
		Similarly, the case $h=h_1=h_2>0$ can be addressed in this manner.
		\hfill $\triangledown\triangledown\triangledown$
	\end{remark}

	\begin{remark} 
		By introducing the augmented state vector, see \cite{Fridman2014},
		\begin{align*}
			x_\text{aug}(k)=\begin{bmatrix} x^\top(k) & x^\top(k-1) & \ldots & x^\top(k-h_1) \end{bmatrix}^\top, 
		\end{align*} 
		it is possible to obtain an augmented non-delayed system dynamics corresponding to \eqref{Eq:System2}, namely
		\begin{align}
			\begin{split} 
				&x_\text{aug}(k+1)=A_\text{aug}x_\text{aug}(k)+B_{\text{aug}}u(k-h_2),\\
				&k\in\mathbb{Z}_{\geq 0},\quad x_{\text{aug}}(k)\in\R^{(h_1+1)n},\quad u(k)\in\R^m,
			\end{split} 
			\label{syssta}
		\end{align} 
		with
		\begin{align}
			\begin{split} 
				A_\text{aug}&=\begin{bmatrix}
					A_0 & 0 & \ldots & A_1\\
					I_n & 0 & \ldots & 0\\
					\vdots & \ddots & \ddots & \vdots\\
					0 & \ldots & I_n & 0
				\end{bmatrix}, \; B_{\text{aug}}=\begin{bmatrix} B \\ 0\\\vdots \\ 0\end{bmatrix}.
			\end{split} 
			\label{ABaug}
		\end{align}
		In principle, this augmented dynamics could also be used to derive a data-driven formula suitable for control design, e.g. by using the results from \cite{Persis2019}. However, this would require that
		\begin{align*} 
			\text{rank} \begin{bmatrix}U_{h_2,\{0\}}\\ X_{\text{aug},\{0\}} \end{bmatrix} = (h_1+1) n +m.
		\end{align*} 
		Clearly, this can be very demanding in the reasonable scenario that $h_1\gg0$ and also seems unnecessary since only $A_0$ and $A_1$ in $A_\text{aug}$ and $B$ in $B_{\text{aug}}$ are unknown, see \eqref{ABaug}. In addition, practically meaningful scenarios in which the delay becomes uncertain (and possibly time-varying) in closed-loop cannot be addressed with the augmented dynamics, see also the discussion below Assumption~\ref{Assum:Main}.
		\hfill $\triangledown\triangledown\triangledown$
	\end{remark}
	
	\begin{remark}
		
		In view of the representation \eqref{ABaug}, Willems' Lemma \cite[Cor.~2]{Willems2005} provides sufficient conditions for guaranteeing $\mathrm{rank}(W_{0})=m+2n$ with $W_{0}$ as in \eqref{Eq:W0}, i.e., the controllability of the pair $(A_{\mathrm{aug}},B_{\mathrm{aug}})$ and the persistence of excitation of order $(h_1+1)n+1$ of the input sequence.
		
		\hfill $\triangledown\triangledown\triangledown$
	\end{remark}
	
\section{Data-Driven Formulas for Controlling Linear Discrete-Time Time-Delay Systems}
\label{Sec:ControlFormulas}
	This section is dedicated to the derivation of data-based controller synthesis formulas for the system \eqref{Eq:System} in the presence of 
	uncertain, time-varying, bounded input and state delays $h_1(k)<\bar h$ and $h_2(k)<\bar h$, respectively. The main tool to achieve this goal is Lemma \ref{Lem:ClosedLoopRep} together with the recorded data sequences $u_{[-\bar{h},T]}$ and $x_{[-\bar{h},T]}$. 
	More precisely, three goals are pursuit in this section for the system \eqref{Eq:System}:
	\begin{enumerate}
		\item	Design of a feedback gain for system stabilization.
		\item	Design of a feedback gain, which ensures a prescribed cost for the input and state trajectory.
		\item	Design of a feedback gain, which ensures a prescribed $L_2$-gain of the system with respect to additive disturbances.
	\end{enumerate}
	With regard to item 3), we note that in the present setting the $H_\infty$ control design is performed in the time domain by using the $L_2$-gain, which we recall is defined as the maximum energy amplification ratio of the system \cite{Fridman2014}. Also, as discussed in Section~\ref{Sec:Preliminaries}, we account for the event that the delays $h_1(k)$ and $h_2(k)$ may become uncertain during the operation of the closed-loop system. 
	
	We start with the first item, i.e., system stabilization, which not only is the simplest scenario, but also paves the path for finding solutions to the other two items. Hence, formally the first problem we address is the following.
		
	\begin{problem}
		\label{Prob:Stabilization}
		Consider the system \eqref{Eq:System} with given $\bar{h}>0$ an upper bound for the input and state delays. Find a feedback gain $K$, such that the origin of system \eqref{Eq:System} with feedback $u(k)=Kx(k)$ is an asymptotically equilibrium point for all uncertain delays $h_1(k)\in[0,\bar{h}]$ and $h_2(k)\in[0,\bar{h}]$.		
		\hfill $\triangledown\triangledown\triangledown$
	\end{problem}

	In an analogous fashion to the non-delayed case \cite{Persis2019}, also in the present delayed setting the matrix $G_K$ in \eqref{Eq:DataRepCL} plays the role of a decision variable in a direct data-driven controller synthesis. By exploiting this fact together with the closed-loop data representation given in Lemma~\ref{Lem:ClosedLoopRep}, we provide the following solution to Problem \ref{Prob:Stabilization}.
	
	\begin{theorem}[Stabilization with Static State Feedback]
		\label{Thm:Stabilization}
		Consider the system \eqref{Eq:System} and suppose that $\mathrm{rank}(W_0)=m+2n$ with $W_0$ as in \eqref{Eq:W0}. Given a positive delay bound $\bar{h}$ and a tuning parameter $\varepsilon>0$, let there exist $n\times n$ matrices $\bar{P}>0$, $\bar{S}>0$, $\bar{R}_i>0$, $\bar{S}_{12,i}$, with $i=\{1,2\}$, and $T\times n$ matrices $Q_1$, $Q_2$ and $Q_3$ such that
		\begin{gather}
			\bar{\mathbf{\Phi}}>0, \label{Eq:Thm01a}\\
			\begin{bmatrix}
			\bar{R}_i & \bar{S}_{12,i}\\ \star & \bar{R}_i
			\end{bmatrix}\geq 0, \label{Eq:Thm01b}
		\end{gather}
		with $\bar{\mathbf{\Phi}}$ given in \eqref{Eq:barBFPhi} on p.~5 and
		\begin{align}
						\begin{split}
			U_{h_2,\{0\}}Q_2 &= U_{h_2,\{0\}}Q_3=0,\\
			X_{h_1,\{0\}}Q_1 &= X_{h_1,\{0\}}Q_3=0,\\
			X_{\{0\}}Q_1 &= X_{\{0\}}Q_2=0,\\
			X_{h_1,\{0\}}Q_2 &= X_{\{0\}}Q_3.
			\end{split}\label{Eq:QRestrictions}
		\end{align}	
		
		Choose the feedback gain as
		\begin{align}
			K = U_{h_2,\{0\}}Q_1\Big(X_{\{0\}}Q_3\Big)^{-1}. \label{Eq:Thm01control}
		\end{align}
		Then for all delays $h_1(k)\in [0,\bar{h}]$ and $h_2(k)\in [0,\bar{h}]$ for all $k\in\mathbb{Z}_{\geq 0}$, the origin of \eqref{Eq:System} in closed-loop with the control $u(k)=Kx(k)$ is asymptotically stable.		
		\hfill $\triangledown\triangledown\triangledown$
	\end{theorem}

	\begin{figure*}[t!]
		\begin{align}
			\bar{\mathbf{\Phi}}=\left[\begin{array}{c c c c c}
			\bar{\mathbf{\Phi}}_{11} & \bar{\mathbf{\Phi}}_{12} & \bar{\mathbf{\Phi}}_{13} & -\bar{S}_{12,1}-\bar{S}_{12,2} & \bar{\mathbf{\Phi}}_{15}\\
				\star & 2\bar{R}_1-\bar{S}_{12,1}-\bar{S}^\top_{12,1} & 0 & -\bar{R}_1+\bar{S}_{12,1} & -\varepsilon \big(X_{\{1\}}Q_2\big)^\top\\
				\star & \star & 2\bar{R}_2-\bar{S}_{12,2}-\bar{S}^\top_{12,2} & -\bar{R}_2+\bar{S}_{12,2} & -\varepsilon \big(X_{\{1\}}Q_1\big)^\top\\
				\star & \star & \star & \bar{R}_1+\bar{R}_2+\bar{S} & 0\\
				\star & \star & \star & \star & \bar{\mathbf{\Phi}}_{55}
			\end{array}\right], \label{Eq:barBFPhi}
		\end{align}
		\vspace*{-10pt}
		\begin{gather*}
			\bar{\mathbf{\Phi}}_{11}=\bar{P}-\bar{S}+(1-\bar{h}^2)(\bar{R}_1+\bar{R}_2)-X_{\{1\}}Q_3-\big(X_{\{1\}}Q_3\big)^\top,\;\bar{\mathbf{\Phi}}_{12}=-\bar{R}_1+\bar{S}_{12,1}-X_{\{1\}}Q_2,\\
			\bar{\mathbf{\Phi}}_{13}=-\bar{R}_2+\bar{S}_{12,2}-X_{\{1\}}Q_1,\; \bar{\mathbf{\Phi}}_{15}=\bar{h}^2(\bar{R}_1+\bar{R}_2)+X_{\{0\}}Q_3-\varepsilon\left(X_{\{1\}}Q_3\right)^\top,\\
			\bar{\mathbf{\Phi}}_{55}=-\bar{P}-\bar{h}^2(\bar{R}_1+\bar{R}_2)+\varepsilon\big(X_{\{0\}}Q_3+Q^\top_{3}X^\top_{\{0\}}\big).
		\end{gather*}
	\end{figure*}
	
	Note that $\bar{\Phi}>0$ in \eqref{Eq:barBFPhi} implies that $\bar{\Phi}_{55}>0$, meaning that $X_{\{0\}}Q_3+Q^\top_3X^\top_{\{0\}}>0$, i.e., $X_{\{0\}}Q_3$ is nonsingular. To build the inequalities \eqref{Eq:Thm01a} and \eqref{Eq:Thm01b} only the recorded data from the sequences $x_{[-\bar{h},T]}$ and $u_{[-\bar{h},T]}$ is needed. Once the matrices $Q_1$, $Q_2$ and $Q_3$ are found such that \eqref{Eq:Thm01a}, \eqref{Eq:Thm01b} and \eqref{Eq:QRestrictions} hold, the feedback gain $K$ can be computed directly from \eqref{Eq:Thm01control}. In this way, the process of identifying the system matrices and the a posteriori controller design is combined into a single direct data-driven synthesis step.
	
	Differently from the non-delay case, see e.g., \cite{Persis2019}, in the present setting the linear matrix inequalities (LMIs) \eqref{Eq:Thm01a} and \eqref{Eq:Thm01b} are accompanied by the equality constraints \eqref{Eq:QRestrictions}. The reason for this lies in the closed-loop representation \eqref{Eq:DataRepCL} and in particular \eqref{Eq:GK}. To see this, consider the right hand-side of \eqref{Eq:GK}. Not only the matrices in the main diagonal are needed to obtain \eqref{Eq:DataRepCL}, but also the zeros, which gives rise to the equality constraints \eqref{Eq:QRestrictions}. This does not happen in the non-delay case since there the closed-loop system is fully described by the single matrix $A+BK$, while in the present case three separated matrices are required.
	
	Once a solution for Problem \ref{Prob:Stabilization} is given, we can think of including performance criteria in the controller design. For linear systems, it is common to attempt the minimization of the system trajectories and the control effort. This results in a linear quadratic regulator (LQR) design. However, for systems of the form (II.1) an optimal control gain does not exist due to the uncertain delays \cite[Sec.~6.2.3]{Fridman2014}. Instead, one can attempt to find a feedback gain which guarantees a certain cost. This yields the problem formulation below.

	\begin{problem}
		\label{Prob:GuaranteedCostControl}
		Consider the system \eqref{Eq:System} with $x(0)=x_0$ and $x(k)=0$ for $k<0$ with cost function
		\begin{align}
			J&=\sum^{\infty}_{j=0}z^\top(k)z(k),
			\label{Eq:CostFunction}
		\end{align}
		and performance output 
		\begin{align}
			z(k)=L_1x(k)+L_2x_{h_{1}(k)}(k)+Du_{h_{2}(k)}(k), \label{Eq:PerfOutput}
		\end{align} 
		with $z(k)\in\mathbb{R}^q$ and constant matrices $L_1\in\mathbb{R}^{q\times n}$, $L_2\in\mathbb{R}^{q\times n}$ and $D\in\mathbb{R}^{q\times m}$.
		Given a cost $\delta>0$, find a feedback gain $K$ that guarantees  $J\leq\delta$ for all uncertain delays $h_1(k)\in[0,\bar{h}]$ and $h_2(k)\in[0,\bar{h}].$	
		\hfill $\triangledown\triangledown\triangledown$
	\end{problem}

	The result provided in Theorem \ref{Thm:Stabilization} can be extended to address Problem \ref{Prob:GuaranteedCostControl} by including the effect of the cost $\delta$ and the functional $J$ in the inequalities \eqref{Eq:Thm01a} and \eqref{Eq:Thm01b}. By doing so, we obtain the following result.
	
	\begin{corollary}[Guaranteed Cost Control]
		\label{Cor:GuaranteedCostControl}
		Consider the system \eqref{Eq:System} together with the considerations presented in Problem \ref{Prob:GuaranteedCostControl}. Suppose that $\mathrm{rank}(W_0)=m+2n$ with $W_0$ as in \eqref{Eq:W0}. Given a positive delay bound $\bar{h}$, the cost $\delta>0$ and a tuning parameter $\varepsilon>0$, let there exist $n\times n$ matrices $\bar{P}>0$, $\bar{S}>0$, $\bar{R}_i>0$, $\bar{S}_{12,i}$, with $i=\{1,2\}$, and $T\times n$ matrices $Q_1$, $Q_2$ and $Q_3$ such that
		\begin{gather}
			\bar{\mathbf{\Psi}}=\left[\begin{array}{c c}\bar{\mathbf{\Phi}} & -\mathbf{\kappa}^\top\\ \star & I_q\end{array}\right]>0 \label{Eq:barBFPsi}\\
			\mathbf{\kappa}= \begin{bmatrix}L_1X_{\{0\}}Q_3 & L_2X_{\{0\}}Q_3 & DU_{h_2,\{0\}}Q_1 & 0 & 0\end{bmatrix},\nonumber\\
			\begin{bmatrix}
			\bar{R}_i & \bar{S}_{12,i}\\ \star & \bar{R}_i
			\end{bmatrix}\geq 0, \label{Eq:Cor01b}
		\end{gather}
		together with \eqref{Eq:QRestrictions} are satisfied with $\bar{\mathbf{\Phi}}$ given in \eqref{Eq:barBFPhi}, in addition to
		\begin{gather}
			\begin{bmatrix}\delta & -x^\top_0\\ \star & X_{\{0\}}Q_3+Q^\top_3X^\top_{\{0\}}-\bar{P}\end{bmatrix}>0. \label{Eq:Cor01c}
		\end{gather}		
		Choose the feedback gain
		\begin{align}
			K = U_{h_2,\{0\}}Q_1\Big(X_{\{0\}}Q_3\Big)^{-1}. \label{Eq:Cor01control}
		\end{align}
		Then for all delays $h_1(k)\in[0,\bar{h}]$ and $h_2(k)\in[0,\bar{h}]$ for all $k\in\mathbb{Z}_{\geq 0}$, the origin of \eqref{Eq:System} in closed-loop with the control $u(k)=Kx(k)$ is exponentially stable. Furthermore, this control ensures a guaranteed cost $\delta$ for $J$ given in \eqref{Eq:CostFunction}, i.e., $J\leq\delta$.		
		\hfill $\triangledown\triangledown\triangledown$
	\end{corollary}

	Another way of introducing performance criteria into the control design is to consider external disturbances affecting the system and to impose restrictions to the response of the system subject to these disturbances. 
   For linear time invariant systems, this is usually done by minimizing the $H_\infty$ norm of the system. In the present setting, the $H_\infty$ control design is performed in the time domain by using the $L_2$-gain. The resulting control problem is formalized as follows.

	\begin{problem}
		\label{Prop:HinftyControl}
		Consider the system \eqref{Eq:System} with an additive disturbance $\omega(k)\in\mathbb{R}^p$ and feedback gain $K$, i.e., 
		\begin{align}
			x(k+1)= A_0x(k)&+A_1x_{h_{1}(k)}(k)\nonumber\\
				&+BKx_{h_{2}(k)}(k)+D_0\omega(k), \label{Eq:PerturbedSys}
		\end{align}
		together with the performance output
		\begin{align}
			z(k)=L_1x(k)+L_2x_{h_{1}(k)}(k)+DKx_{h_{2}(k)}(k), \label{Eq:PerfOutputH}
		\end{align} 
		with $z(k)\in\mathbb{R}^q$ and constant matrices $L_1\in\mathbb{R}^{q\times n}$, $L_2\in\mathbb{R}^{q\times n}$, $D_0\in\mathbb{R}^{n\times p}$ and $D\in\mathbb{R}^{q\times m}$. Fix a constant $\gamma>0$. For all uncertain delays $h_1(k)\in[0,\bar{h}]$ and $h_2(k)\in[0,\bar{h}]$, find a feedback gain $K$ such that, for $\omega(k)=0$, the origin of \eqref{Eq:PerturbedSys} is an asymptotically stable equilibrium point and for $\omega(k)\neq 0$, the system \eqref{Eq:PerturbedSys} has an $L_2$-gain less than $\gamma$.
		
		\hfill $\triangledown\triangledown\triangledown$
	\end{problem}

	As before, it is possible to give a solution to Problem \ref{Prop:HinftyControl} by extending Theorem \ref{Thm:Stabilization} and including the required conditions in the data-based inequalities \eqref{Eq:Thm01a} and \eqref{Eq:Thm01b}. The following results is consistent with such approach.	

	\begin{corollary}[Static $H_\infty$ Control]
		\label{Cor:HinftyControl}
		Consider the system \eqref{Eq:PerturbedSys} together with the considerations given in Problem \ref{Prop:HinftyControl}. Suppose that $\mathrm{rank}(W_0)=m+2n$ with $W_0$ given in \eqref{Eq:W0}. Given positive constants $\bar{h}$ and $\gamma$ and the tuning parameter $\varepsilon>0$, suppose that there exists $n\times n$ matrices $\bar{P}>0$, $\bar{S}>0$, $\bar{R}_i>0$ and $\bar{S}_{12,i}>0$, for $i=\{1,2\}$, and $T\times n$ matrices $Q_1$, $Q_2$ and $Q_3$ such that the following data-based inequalities are satisfied.
		\begin{gather}
			\bar{\mathbf{\Gamma}}=\left[\begin{array}{c c}\bar{\mathbf{\Psi}} & \mathbf{\kappa}^\top_2\\ \star & \gamma I_p\end{array}\right]>0, \label{Eq:barBFGamma}\\
			\mathbf{\kappa}_2=\begin{bmatrix}-D^\top_0 & 0 & 0 & 0 & -\varepsilon D^\top_0 & 0\end{bmatrix},\nonumber\\
			\begin{bmatrix}
			\bar{R}_i & \bar{S}_{12,i}\\ \star & \bar{R}_i
			\end{bmatrix}\geq 0, \label{Eq:Cor02b}
		\end{gather}
		together with \eqref{Eq:QRestrictions}, and where $\bar{\mathbf{\Psi}}$ is given in \eqref{Eq:barBFPsi}. Choose the feedback gain
		\begin{align}
			K = U_{h_2,\{0\}}Q_1\Big(X_{\{0\}}Q_3\Big)^{-1}. \label{Eq:Cor02control}
		\end{align}
		Then for all delays $h_1(k)\in[0,\bar{h}]$ and $h_2(k)\in[0,\bar{h}]$, the origin of \eqref{Eq:PerturbedSys} is an asymptotically stable equilibrium point for $\omega(k)=0$. Furthermore, for $\omega(k)\neq 0$ the system \eqref{Eq:PerturbedSys} has an $L_2$-gain less than $\gamma$.		
		\hfill $\triangledown\triangledown\triangledown$
	\end{corollary}

\section{Handling Unknown Constant Delays and Noisy Data}
\label{Sec:Uncertainties}
	In this section we address the problems that the delays $h_1$ and $h_2$ of the system \eqref{Eq:System} are unknown and that the available data is corrupted by noise. To this end, we assume that the sequences $x_{[-\bar{h},T]}$ and $u_{[-\bar{h},T]}$ can be expressed as the sum of two sequences:
	\begin{align}
		\begin{split}
			x_{[-\bar{h},T]}&=x^{\text{nom}}_{[-\bar{h},T]}+x^{\delta}_{[-\bar{h},T]},\\
			u_{[-\bar{h},T]}&=u^{\text{nom}}_{[-\bar{h},T]}+u^{\delta}_{[-\bar{h},T]},
		\end{split} \label{Eq:NoisyData}
	\end{align}
	where the superscript `$\text{nom}$' denotes the sequence that corresponds to the dynamics \eqref{Eq:System2}, i.e., the \emph{nominal} part of the data, whereas the superscript `$\delta$' denotes the sequence corresponding to the measurement noise.

	Hence, the objectives of this section are to provide formulas for determining the delays $h_1$ and $h_2$ from the recorded data and to robustify the design of the feedback gains from Section~\ref{Sec:ControlFormulas} with respect to additive noise.
	
	\subsection{Data-Based System Representations for Unknown Constant Delays}
		By using the sequences $u_{[-\bar{h},T]}$ and $x_{[-\bar{h},T]}$ as described in \eqref{Eq:NoisyData}, we can build the matrix $W_0$ as in \eqref{Eq:W0}. Since the construction of $W_0$ is linear, it is possible to split it in two parts, one corresponding to the data generated by the system \eqref{Eq:System2}, i.e. the nominal ('nom') data, and one to the noise, i.e.,
		\begin{align}
			W_0=W^{\text{nom}}_0+W^{\delta}_0. \label{Eq:PerturbedW0}
		\end{align}
		However, $W_0$ in \eqref{Eq:PerturbedW0} will not result in useful data for arbitrary $W^{\delta}_0$. To study when $W_0$ retains the system information, let $U\in\mathbb{R}^{(m+2n)\times(m+2n)}$ and $V\in\mathbb{R}^{T\times T}$ be orthonormal matrices such that $$\mathrm{Range}(U)=\mathrm{Range}(W_0),\quad \mathrm{Range}(V)=\mathrm{Range}(W^\top_0).$$ Consider the factorization \cite[Sec.~2]{Stewart1977}
		\begin{align}
			\begin{split}
				W_0 &= U\begin{bmatrix}W^{11}_0 & 0\end{bmatrix}V^\top,\\
				W^{\delta}_0 &= U\begin{bmatrix}W^{\delta,11}_0 & W^{\delta,12}_0\end{bmatrix}V^\top,
			\end{split} \label{Eq:W0decomp}
		\end{align}
		where $W^{11}_0\in\mathbb{R}^{(m+2n)\times(m+2n)}$,  $W^{\delta,11}\in\mathbb{R}^{(m+2n)\times(m+2n)}$ and $W^{\delta,12}_0\in\mathbb{R}^{(m+2n)\times T}$. By using the factorization \eqref{Eq:W0decomp}, we introduce the following assumptions related to the impact of the noise. 
		
		\begin{assumption}
			\label{Assum:Noise}
			\phantom{M}
			\begin{enumerate}
				\item 	$\mathrm{rank}(W_0)=\mathrm{rank}(W^{\mathrm{nom}}_0)=m+2\,n$.
				\item	$\big\|W^{\delta,11}_0\big\|_2\big\|W^{\dagger}_0\big\|_2=\big\|W^{\delta,11}_0\big\|_2\big\|\big(W^{11}_0\big)^{-1}\big\|_2<1$.
			\end{enumerate}
		\end{assumption}	
		Assumption \ref{Assum:Noise}.1) implies that $W^{\delta}_0$ does not modify the rank of $W^{\mathrm{nom}}_0$, whereas Assumption \ref{Assum:Noise}.2) restricts the size of the perturbation term $W^{\delta}_0$. The premises of Assumption \ref{Assum:Noise} ensure that $W_0$ is an \emph{acute} perturbation of $W^{\mathrm{nom}}_0$ \cite{Stewart1977,Wedin1973}.

		Now, in order to identify the system delays, consider the matrices
		\begin{align}
			W_{0,(i,j)}=\left[\begin{array}{c}U_{i,\{0\}}\\ X_{j,\{0\}}\\ X_{\{0\}}\end{array}\right], \label{Eq:PerturbedW0b}
		\end{align}
		for $i=\{0,1,\cdots,\bar{h}\}$, $j=\{0,1,\cdots,\bar{h}\}$, where $U_{i,\{0\}}$ is built with $u(k-i)$, $X_{j,\{0\}}$ with $x(k-j)$, and $\bar{h}\geq 0$ is the upper bound for the delays. In the unperturbed case, i.e., for $x^{\delta}_{[-\bar{h},T]}=0$ and $u^{\delta}_{[-\bar{h},T]}=0$, one can verify the next rank conditions in order to determine the system delays 
		\begin{align}
			\mathrm{rank}\big(W_{0,(i,j)}\big)=\mathrm{rank}\left(\left[\begin{array}{c}W_{0,(i,j)}\\ \hline X_{\{1\}}\end{array}\right]\right)=m+2n, \label{Eq:DelayIdentNom}
		\end{align}
		for all $i$ and $j$ in $\{0,1,\cdots,\bar{h}\}$. If for some pair $(i^\star,j^\star)$ the condition above holds, then one can take $h_1=j^\star$ and $h_2=i^\star$ since $X_{\{1\}}$ belongs to the row space of $W_{0,(i^\star,j^\star)}$. If for two or more pairs $(i,j)$ the condition \eqref{Eq:DelayIdentNom} holds, then it is not possible to identify the delays from the recorded data. However, in the perturbed case the condition \eqref{Eq:DelayIdentNom} might never hold due to the effect of the noise. Therefore, instead of \eqref{Eq:DelayIdentNom}, we propose to use the orthogonal distance of $X_{\{1\}}$ to the row space of each $W^{(i,j)}_{0}$ in order to determine the delays. This yields the next proposition, for the presentation of which we introduce the matrix
		\begin{align}
			X_{\{1\}}=X^{\mathrm{nom}}_{\{1\}}+X^{\delta}_{\{1\}},
			\label{Eq:X1per}
		\end{align}
		where $X^{\mathrm{nom}}_{\{1\}}$ denotes the part of the data that corresponds to the dynamics of \eqref{Eq:System2} and $X^{\delta}_{\{1\}}$ denotes the part corresponding to the noise. In addition, we define the orthogonal distance	
			\begin{align}
				d_{(i,j)}\Big(X_{\{1\}}\Big):=\left\|X_{\{1\}}\left(I_T-\left(W_{0,(i,j)}\right)^{\dagger}W_{0,(i,j)}\right)\right\|_2, \label{Eq:OrthDis}
			\end{align}		
			and the function $\psi:\mathbb{R}_{\geq 0}\to\mathbb{R}_{\geq 0}$
			\begin{align}
				\psi(\sigma)=\frac{\sigma}{\left[1+\sigma^2\right]^{1/2}}. \label{App:DistAuxpsi}
			\end{align}
			In addition, upper bounds for the noisy matrices are required. These are represented by the positive, data-dependent constants $r_{X^{\delta}_{\{1\}}}$,  $r_{W^{\delta,11}_{0,(i,j)}}$ and $r_{W^{\delta,12}_{0,(i,j)}}$ satisfying
			\begin{align}
				\begin{split}
					&	\big\|X^{\delta}_{\{1\}}\big\|_2 \leq r_{X^{\delta}_{\{1\}}},\quad 				\big\|W^{\delta,12}_{0,(i,j)}\big\|_2 \leq r_{W^{\delta,12}_{0,(i,j)}},\\
					&	\big\|W^{\delta,11}_{0,(i,j)}\big\|_2 \big\|W^{\dagger}_{0,(i,j)}\big\|_2\leq r_{W^{\delta,11}_{0,(i,j)}}\big\|W^{\dagger}_{0,(i,j)}\big\|_2 < 1,
				\end{split}	\label{Eq:PropDelay01}
			\end{align}
			with $X^{\delta}_{\{1\}}$ and $W^{\delta}_{0,(i,j)}$ as in \eqref{Eq:X1per} and \eqref{Eq:PerturbedW0}, respectively, and $W^{\delta,11}_{0,(i,j)}$ and $W^{\delta,12}_{0,(i,j)}$ as in \eqref{Eq:W0decomp}. Clearly,  $r_{W^{\delta,11}_{0,(i,j)}}$ exists for any data set consistent with Assumption V.1.2).

		\begin{proposition}[Data-Based System Representations for Unknown Delays]
			\label{Prop:DelayIdent}
			Consider the system~\eqref{Eq:System} with corresponding perturbed data sequences $x_{[-\bar{h},T]}$ and $u_{[-\bar{h},T]}$ as introduced in \eqref{Eq:NoisyData}. Let $\bar{h}>0$ be the upper bound for the constant but unknown state and input delays $h_1\in\mathbb{Z}_{\geq 0}$ and $h_2\in\mathbb{Z}_{\geq 0}$. Build the matrices $W_{0,(i,j)}$ as in \eqref{Eq:PerturbedW0b} for $i$ and $j$ in $\{0,1,\cdots,\bar{h}\}$, and suppose that, for all $i$ and $j$, Assumption \ref{Assum:Noise} holds for each of the matrices $W_{0,(i,j)}$ and the respective perturbation $W^{\delta}_{0,(i,j)}$. Furthermore, suppose that the constants $r_{X^{\delta}_{\{1\}}}$,  $r_{W^{\delta,11}_{0,(i,j)}}$ and $r_{W^{\delta,12}_{0,(i,j)}}$ defined in \eqref{Eq:PropDelay01} are known.

			Recall the orthogonal distance $d_{(i^\star,j^\star)}\big(X_{\{1\}}\big)$ given in \eqref{Eq:OrthDis} and the function $\psi(\cdot)$ defined in \eqref{App:DistAuxpsi}. If
			\begin{align}
				d_{(i^\star,j^\star)}\big(X_{\{1\}}\big)\leq  &r_{X^{\delta}_{\{1\}}}+
				\big(\|X_{\{1\}}\|_2+r_{X^{\delta}_{\{1\}}}\big)\cdot \psi\left(\sigma^\star\right)\label{Eq:DelayCond}
			\end{align}	
			where
			\begin{align}
			\sigma^\star=&\frac{r_{W^{\delta,12}_{0,(i^\star,j^\star)}}\left\|W^{\dagger}_{0,(i^\star,j^\star)}\right\|_2}{1-r_{W^{\delta,11}_{0,(i^\star,j^\star)}}\left\|W^{\dagger}_{0,(i^\star,j^\star)}\right\|_2},\label{sigmastar}
			\end{align}	
			for only one pair $(i^\star,j^\star)$, then $h_1=j^\star$ and $h_2=i^\star$. Moreover, the corresponding open- and closed-loop data-based representations are obtained via Proposition \ref{Prop:OpenLoopRep} and Lemma \ref{Lem:ClosedLoopRep}, respectively, by using the matrices $U_{i^\star,\{0\}}$, $X_{j^\star,\{0\}}$ together with $X_{\{0\}}$ and $X_{\{1\}}$.
						
			If condition \eqref{Eq:DelayCond} holds for two or more pairs $(i,j)$, then the delays are not decidable from the available data. 			
			\hfill $\triangledown\triangledown\triangledown$
			\normalcolor
		\end{proposition}
		
		Proposition~\ref{Prop:DelayIdent} provides a tool for deriving data-based system representations for {\em unknown} delays, even in the presence of noise. In addition, the delays themselves are also determined. 	
		
		In general, $(\bar{h}+1)^2$ evaluations of \eqref{Eq:DelayCond} are required. This number reduces to $\bar{h}+1$ if, for example, $h_1=h_2$ or if one of the delays is known. Furthermore, the same idea can be used to identify time-dependent delays. However, in such case, $(\bar{h}+1)^T$ evaluations are required, which might not be computationally feasible.
	
	\subsection{Stabilization with Noisy Data}
		Now we proceed to analyze the impact of noisy data on the controller synthesis formulas derived in Section~\ref{Sec:ControlFormulas}. The main objective is to extend the result of Theorem \ref{Thm:Stabilization} to incorporate a criterion to ensure closed-loop stability of the system \eqref{Eq:System} even when the feedback gain $K$ is computed with corrupted data. In order to account for the impact of the noise in the data, consider the following matrix
		\begin{align}
			\Delta_{[\cdot]}:=\begin{bmatrix}B & A_1 & A_0\end{bmatrix}W^{\delta}_{0}-X^{\delta}_{\{1\}}. \label{Eq:BoundNoise}
		\end{align}
		The quantity $\|\Delta_{[\cdot]}\|_2$ is a measurement of how far the noise is from being a system trajectory. If the noise would correspond to a system trajectory, then it would not affect any of the calculations; though in such case it might not be classified as noise. Therefore, it is logical that only $\Delta_{[\cdot]}$ has an impact on the computation of $K$. By using this measurement of the noise, it is possible to account for it in the feedback  design. This approach yields the next result.
		
		\begin{theorem}[Stabilization with Noisy Data]
			\label{Thm:GainNoise}
			Consider the premises of Theorem \ref{Thm:Stabilization}. Let the recorded data be corrupted by noise as in \eqref{Eq:NoisyData}. Suppose that $\Delta_{[\cdot]}$ in \eqref{Eq:BoundNoise} is bounded as $\|\Delta_{[\cdot]}\|_2\leq \alpha$, with $\alpha>0$ known. Given a positive delay bound $\bar{h}$ and a tuning parameter $\varepsilon>0$, let there exist $n\times n$ matrices $\bar{P}>0$, $\bar{S}>0$, $\bar{R}_i>0$, $\bar{S}_{12,i}$, with $i=\{1,2\}$, $T\times n$ matrices $Q_1$, $Q_2$, $Q_3$, and $\lambda>0,$ such that
			\begin{gather}
				\left[\begin{array}{c c}\bar{\mathbf{\Phi}}-\alpha^2\lambda I_{5n} & \mathbf{Q}^\top\\ \mathbf{Q} & \lambda I_{5T}\end{array}\right]>0, \label{Eq:NoiseThm01a}\\
				\begin{bmatrix}
				\bar{R}_i & \bar{S}_{12,i}\\ \star & \bar{R}_i
				\end{bmatrix}\geq 0, \label{Eq:NoiseThm01b}
			\end{gather}
			together with \eqref{Eq:QRestrictions} hold, where $\bar{\mathbf{\Phi}}$ is given in \eqref{Eq:barBFPhi} and 
			\begin{align}
				\mathbf{Q}=\begin{bmatrix}Q_3 & Q_2 & Q_1 & 0 & 0\\ 0 & 0 & 0 & 0 & 0\\ 0 & 0 & 0 & 0 & 0\\ 0 & 0 & 0 & 0 & 0\\ \varepsilon Q_3 & \varepsilon Q_2 & \varepsilon Q_1 & 0 & 0\end{bmatrix}. \label{Eq:bfQ}
			\end{align}
			Choose the feedback gain
			\begin{align}
				K = U_{h_2,\{0\}}Q_1\Big(X_{\{0\}}Q_3\Big)^{-1}. \label{Eq:NoiseThm01control}
			\end{align}
			Then for all delays $h_1(k)\in[0,\bar{h}]$ and $h_2(k)\in[0,\bar{h}]$ for all $k\in\mathbb{Z}_{\geq 0}$, the origin of \eqref{Eq:System} in closed-loop is asymptotically stable.		
			\hfill $\triangledown\triangledown\triangledown$
		\end{theorem}
		
		For $\alpha=0$, i.e., in the noise free case, the inequality \eqref{Eq:NoiseThm01a} reduces to the one in \eqref{Eq:Thm01a}. As in \eqref{Eq:NoiseThm01a}, the inequalities \eqref{Eq:barBFPsi} and \eqref{Eq:barBFGamma} can be extended to account for data corrupted by noise. Therefore, from Theorem~\ref{Thm:Stabilization} analogous corollaries to Corollary \ref{Cor:GuaranteedCostControl} and Corollary \ref{Cor:HinftyControl} can be derived in a straightforward manner. Hence, their explicit presentation is omitted.

\section{Numerical Example: Unstable Batch Reactor}
\label{Sec:Example}
	
	\begin{table}[tp!]
		\centering	
		\caption{Computed values of the distance $d_{(i,j)}(X_{\{1\}})$ given in \eqref{Eq:OrthDis} for the different values of $i$, with $j=0$ and $T=10$, in the absence of noise with the data recorded from the system \eqref{Eq:ExmSys}.}
		\label{Tab:Distance0}	

		\begin{tabular}{|c || c c c|}
			\hline
			$i$ & $0$ & $1$ & $2$\\ \hline &&&\\[-5pt]
			$d_{(i,0)}(X_{\{1\}})$ & \num{6.85e-4} & \num{2.44e-3} & \num{3.92e-4}\\ \hline\hline
			$i$ & $3$ & $4$ & $5$ \\ \hline &&&\\[-5pt]
			$d_{(i,0)}(X_{\{1\}})$ & \num{2.75e-10}  & \num{2.76e-3} & \num{6.83e-3}\\ \hline\hline
			$i$ & $6$ & $7$ & $8$ \\ \hline &&&\\[-5pt]
			$d_{(i,0)}(X_{\{1\}})$ &  \num{3.88e-3} & \num{1.97e-2} & \num{1.06e-2}\\ \hline\hline
		\end{tabular}
	\end{table}
	
	To exemplify the proposed method, we consider the unstable linearized batch reactor in \cite[pp.~63]{Walsh2001} controlled through a network, and described by the dynamics
	\begin{align}
		\begin{split}
			\dot{x}(t)&=\begin{bmatrix}1.38 & -0.20 & 6.71 & -5.67\\-0.58 & -4.29 & 0 & 0.67\\ 1.06 & 4.27 & -6.65 & 5.89\\ 0.04 & 4.27 & 1.34 & -2.10\end{bmatrix}x(t) \\
				&\qquad\qquad\qquad\quad +\begin{bmatrix}0 & 0\\5.67 & 0\\1.13 & -3.14\\1.13 & 0\end{bmatrix}u(t-h_2).
		\end{split} \label{Eq:ExmSys}
	\end{align}
	The input to the system is generated using a zero-order hold (ZOH) with a sampling time of $10$ [ms]. This sampling time is taken as the base time. Additionally, a constant input delay $h_2=3$ ($30$ [ms]) is introduced in the system \eqref{Eq:ExmSys}. Furthermore, we assume a maximum delay length of $\bar{h}=8$. We consider that the plant has been in operation for a certain time using the PI-control given in~\cite{Walsh2001}. In the context of the present paper, it is assumed that this controller implementation is based on expertise rather than on a model. Furthermore, the plant operation point is assumed  known and corresponds to
	\begin{align}
		\begin{split}
			x_{\text{op}} &= \begin{bmatrix}24.35 & 14 & 48.78 & 63.13\end{bmatrix}^\top,\\
			u_{\text{op}} &= \begin{bmatrix}5.565 & 44.36\end{bmatrix}^\top.
		\end{split} \label{Eq:ExampOP}
	\end{align}
	The overall objective is to stabilize the system around \eqref{Eq:ExampOP}.
	
	\subsection{Scenario 1: Noise-Free Data and Unknown Delay}
	\label{subsec:noisfree}

	For generating the input-state data, and to characterize the system~\eqref{Eq:ExmSys}, we feed the reference $y_{\mathrm{ref}}=[10\;14]^\top$ in combination with the excitation signal $u_{\text{exc}}(t)$ defined below to the PI-control already available in the plant \cite{Walsh2001}.
	\begin{align*}
		u_{\mathrm{exc}}(t)=\begin{bmatrix}10\sin(7\pi\,t)-5\sin(11\pi\,t)\\ 8\sin(9\pi\,t)-6\sin(13\pi\,t) \end{bmatrix}.
	\end{align*}
	The resulting excitation signal is shown in Figure \ref{Fig:ExcU01}. For the control design, we assume $h_2$ constant, but unknown, and given the physical background of the system \eqref{Eq:ExmSys}, we have $h_1=0$.

	As first step, we seek to investigate the value of $h_2$ in the range $\{0,1,\cdots,\bar{h}\}$. Note that for $T=m+n$, which yields a square $W_0$, the distance $d_{(i,j)}(X_{\{1\}})$ defined in \eqref{Eq:OrthDis} is always zero since $I_{n+m}-W_0W^{-1}_0=0$. Therefore, we choose $T=10>n+m$. We identify the length of $h_2$ using the result of Proposition \ref{Prop:DelayIdent}. For the noise free case ($r_{X^{\delta}_{\{1\}}}=r_{W^{\delta,11}_0}=r_{W^{\delta,12}_0}=0$), the criterion given in \eqref{Eq:DelayCond} reads as $d_{(i,j)}(X_{\{1\}})\leq 0$. The resulting values for the distance $d_{(i,j)}(X_{\{1\}})$ for the different values of $i$, with $j=0$ (since $h_1=0$) and $T=10$, are shown in Table \ref{Tab:Distance0}. From Table \ref{Tab:Distance0}, the input delay can be clearly determined as $h_2=3$ since the distance $d_{(3,0)}(X_{\{1\}})$ is practically zero and its value is due to numerical errors.
	
	Now that $h_2$ has been determined, the matrices $X_{\{0\}}$, $X_{\{1\}}$ and $U_{h_2,\{0\}}$ can be built. To illustrate the application of the data-driven controller synthesis from Section \ref{Sec:ControlFormulas}, we consider the stabilization of the system \eqref{Eq:ExmSys} at the operational point \eqref{Eq:ExampOP}. We assume that the network-induced delay takes values in the set $\{0,1,\cdots,5\}$, whereas the input delay remains constant at $h_2=3$. This satisfies the delay upper bound $\bar{h}=8$, which is used in the formulas provided in Theorem \ref{Thm:Stabilization}. By using the data-based matrices $X_{\{0\}}$, $X_{\{1\}}$ and $U_{h_2,\{0\}}$, and following Theorem \ref{Thm:Stabilization}, we solve \eqref{Eq:Thm01a}, \eqref{Eq:Thm01b} and \eqref{Eq:QRestrictions} with $\bar{R}_1=\bar{S}_{12}=0$ and $Q_2=0$ using CVX\footnote{CVX can parse LMIs with equality constraints and process them as a semidefinite program. Therefore, including \eqref{Eq:QRestrictions} is straightforward in this case. For noisy data, a numerically more robust approach consists in jointly minimizing the norms $\|U_{h_2,\{0\}}Q_2\|_2$, $\|U_{h_2,\{0\}}Q_3\|_2$, $\|X_{h_1,\{0\}}Q_1\|_2$, $\|X_{h_1,\{0\}}Q_3\|_2$, $\|X_{\{0\}}Q_1\|_2$, $\|X_{\{0\}}Q_2\|_2$, $\|X_{h_1,\{0\}}Q_2-X_{\{0\}}Q_3\|_2$ subject to the LMIs \eqref{Eq:Thm01a} and \eqref{Eq:Thm01b}, which is a convex problem. If needed, the norm minimization can be transformed into a semidefinite program following \cite{SDP1996}.}\cite{cvx2014}. For this, we used $\varepsilon=3$. This yields the following feedback gain:
	\begin{align}
		\begin{split}
			K &= \begin{bmatrix} 0.813 & -0.282 & 0.115 & -1.121\\ 2.255 & -0.549 & 1.894 & -1.226\end{bmatrix}.
		\end{split} \label{Eq:GainsBR01}
	\end{align}
	To compare our result with a model-based approach, we also computed a stabilizing gain following \cite[Chap.~6]{Fridman2014} by discretizing the batch reactor model in \eqref{Eq:ExmSys} with the given base time of $10$ [ms]. By using the given delay upper bound $\bar{h}=8$ and with $\varepsilon=3$, we obtained the controller gain
	\begin{align}
		K_{\mathrm{MB}} &= \begin{bmatrix}0.338 & -0.511 & -0.081 & -0.626\\ 2.117 & 0.034 & 1.512 & -0.914\end{bmatrix}. \label{Eq:ModelBasedK} 
	\end{align}
	We simulate the stabilization of the system \eqref{Eq:ExmSys} around the operational point \eqref{Eq:ExampOP} for the two gains $K$ and $K_{MB}$ in \eqref{Eq:GainsBR01} and \eqref{Eq:ModelBasedK}, respectively. We used a network induced delay that randomly changed in the proposed range, i.e., between zero and five. In Figure \ref{Fig:ErrorNormBR}, the error norm between the system state and $x_{\text{op}}$ is shown. We can observe that both controllers achieve the task in a similar time, under the same circumstances. Finally, in Figure \ref{Fig:TrajBR}, the response of the system \eqref{Eq:ExmSys} to the control process using $K$ in \eqref{Eq:GainsBR01} is illustrated for reference.
	
	\subsection{Scenario 2: Noisy Data and Unknown Delay}	
	
	In order to evaluate the robustness of the proposed approach under corrupted measurements, we add an uniform distributed random signal $\delta(k)$ to each measurement $x(k)$ and $u(k)$ of the system \eqref{Eq:ExmSys}. The range of $\delta(k)$ corresponds to $[-\num{1e-4},\num{1e-4}]$. As before, for the control design we assume a constant and unknown input delay $h_2$ as well as $h_1=0$, with the same delay upper bound $\bar{h}=8$. For this section, and because we are dealing with data corrupted by noise, we set $T=50$. Now, in order to determine the input delay, and following Proposition \ref{Prop:DelayIdent}, we need to estimate the upper bounds
	\begin{align*}
		r_{X^{\delta}_{\{1\}}}\geq\left\|X^{\delta}_{\{1\}}\right\|_2,\; r_{W^{\delta,11}_0}\geq\left\|W^{\delta,11}_0\right\|_2,\; r_{W^{\delta,12}_0}\geq\left\|W^{\delta,12}_0\right\|_2.
	\end{align*}
	Since the noise follows a uniform distribution, it is bounded in magnitude. We can find the required upper bounds by using the Frobenius norm with the maximum value for each component:
	\begin{align*}
		\left\|X^{\delta}_{\{1\}}\right\|_2&\leq 10^{-4}\cdot \sqrt{n\cdot T}=\num{1.41e-3}=:r_{X^{\delta}_{\{1\}}},\\
		\left\|W^{\delta,11}_0\right\|_2&\leq 10^{-4}(n+m)=\num{6e-4}=:r_{W^{\delta,11}_0},\\
		\left\|W^{\delta,12}_0\right\|_2&\leq 10^{-4}\sqrt{(n+m)(T-m-n)}\\
			& = \num{1.62e-3}=:r_{W^{\delta,12}_0}.
	\end{align*}	
	We proceed to compute the distance $d_{(i,j)}(X_{\{1\}})$ given in \eqref{Eq:OrthDis} and the criterion given in \eqref{Eq:DelayCond}, but with noisy data. The results are shown in Table \ref{Tab:Distance1}. In contrast to the noise-free scenario in Section \ref{subsec:noisfree}, the distance value for $i=3$, i.e., the correct delay length, is not as close to zero as before. Still, using the criterion derived in Proposition \ref{Prop:DelayIdent}, we can correctly identify the input delay as $h_2=3$ since it is the only case in which the criterion \eqref{Eq:DelayCond} is satisfied.
	
	To guarantee a robust closed-loop performance despite the presence of noise, we seek to employ Theorem \ref{Thm:GainNoise} for the controller synthesis.
	Thus in order to proceed, we need to estimate a bound for $\|\Delta_{[\cdot]}\|_2$ in \eqref{Eq:BoundNoise}. From \eqref{Eq:BoundNoise}, we have
	\begin{align*}
		\left\|\Delta_{[\cdot]}\right\|_2\leq \left\|\begin{bmatrix}B & A_0\end{bmatrix}\right\|_2\left\|W^{\delta}_0\right\|_2+\left\|X^{\delta}_{\{1\}}\right\|_2.
	\end{align*}
	Again, using a bound over the Frobenius norm, we obtain
	\begin{align}
		\begin{split}
			\left\|W^{\delta}_0\right\|_2 &\leq 10^{-4}\cdot \sqrt{(m+n)\cdot T}=\num{1.73e-3}.
		\end{split} \label{Eq:ExampAux01}
	\end{align}
	To estimate the norm of the system matrices we use the relation 
	\begin{align}
		&\left\|\begin{bmatrix}B & A_0\end{bmatrix}\right\|_2 =\left\|X^{\mathrm{nom}}_{\{1\}}\left(W^{\mathrm{nom}}_0\right)^{\dagger}\right\|_2 \nonumber\\
		&\leq \left(\left\|X_{\{1\}}\right\|_2+r_{X^{\delta}_{\{1\}}}\right)\frac{\sqrt{2}\|W^{\dagger}_0\|_2}{1-r_{W^{\delta}_0}\|W^{\dagger}_0\|_2}\leq 102, \label{Eq:ExampAux02}
	\end{align}
	where the last step follows from the upper bound for $\|(W^{\mathrm{nom}}_0)^{\dagger}\|_2$ given in \cite[Lem.~3.1]{Wedin1973}. Using \eqref{Eq:ExampAux01} and \eqref{Eq:ExampAux02} we obtain
	\begin{align*}
		\left\|\Delta_{[\cdot]}\right\|_2\leq 0.191.
	\end{align*}
	Thus we have $\|\Delta_{[\cdot]}\|_2\leq \alpha$ with $\alpha=0.191$. Now, we can compute the feedback gain $K$ using Theorem~\ref{Thm:GainNoise} and CVX \cite{cvx2014}, which for $\alpha=0.251$, $\varepsilon=70$ and $\lambda=\num{1.42e3}$ yields
	\begin{align}
		K=\begin{bmatrix} 0.736 & -0.623 & \num{9.17e-2} & -1.07\\ 2.35 & \num{-7.99e-3} & 1.70 & -1.17\end{bmatrix}. \label{Eq:GainsBR02}
	\end{align}
	To test this new feedback gain, we use the same setting as in Section~\ref{subsec:noisfree}, i.e., the stabilization around the operational point $x_{\text{op}}$ in \eqref{Eq:ExampOP} with uncertain network induced delay. The results of this simulation are presented in Figure \ref{Fig:TrajBR2}. As can be seen, despite $K$ being computed using data corrupted by noise, the stabilization is achieved. This demonstrates the robustness of the proposed approach with respect to noisy data.	 
		
	\begin{table}[htp!]
		\centering	
		\caption{Computed values of the distance $d_{(i,j)}(X_{\{1\}})$ given in \eqref{Eq:OrthDis} and the upper bound given in \eqref{Eq:DelayCond} for the different values of $i$, with $j=0$ and $T=50$, with the data recorded from the system \eqref{Eq:ExmSys} corrupted by noise.}
		\label{Tab:Distance1}		
		\begin{tabular}{|c || c c c|}
			\hline
			$i$ & $0$ & $1$ & $2$\\ \hline &&&\\[-5pt]
			$d_{(i,j)}(X_{\{1\}})$ & \num{1.81e0} & \num{1.35e0} & \num{6.71e-1} \\ 
			\eqref{Eq:DelayCond} & \num{1.40e-1} & \num{1.35e-1} & \num{1.30e-1}\\ \hline \hline
			$i$ & $3$ & $4$ & $5$\\ \hline &&&\\[-5pt]
			$d_{(i,j)}(X_{\{1\}})$ & \num{6.88e-4} & \num{5.68e-1} & \num{1.03e0} \\ 
			\eqref{Eq:DelayCond} & \num{1.27e-1} & \num{1.29e-1} & \num{1.36e-1}\\ \hline \hline
			$i$ & $6$ & $7$ & $8$\\ \hline &&&\\[-5pt]
			$d_{(i,j)}(X_{\{1\}})$ & \num{1.40e0} & \num{1.69e0} & \num{2.14e0} \\ 
			\eqref{Eq:DelayCond} & \num{1.56e-1} & \num{1.81e-1} & \num{1.50e-1}\\ \hline 
		\end{tabular}
	\end{table} 

	\begin{figure}[htp!]
		\centering
		\includegraphics[width=0.47\textwidth]{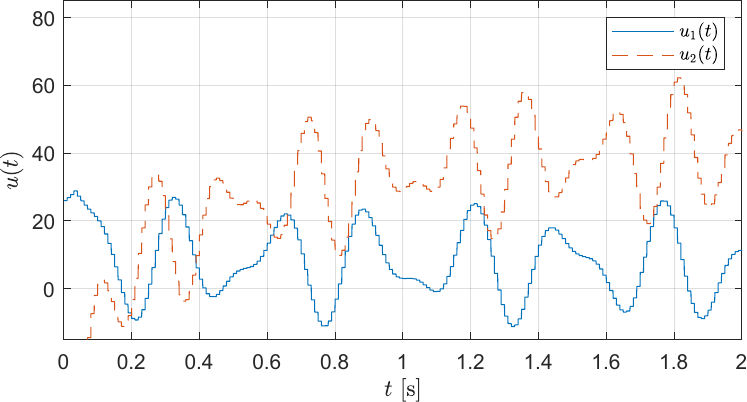}
		\caption{Input signal $u(t)$ used to excite the batch reactor in \eqref{Eq:ExmSys}.}
		\label{Fig:ExcU01}
	\end{figure}
	
	\begin{figure}[htp!]
		\centering
		\includegraphics[width=0.47\textwidth]{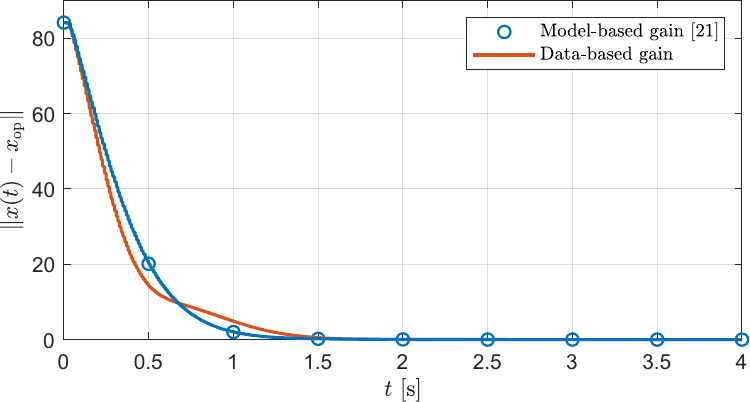}
		\caption{Norm of the error between the state of the batch reactor \eqref{Eq:ExmSys} and the operational point $x_{\text{op}}$ in \eqref{Eq:ExampOP} for the control gain $K$ in \eqref{Eq:GainsBR01} computed using data and the gain $K_{\mathrm{MB}}$ in \eqref{Eq:ModelBasedK} computed following \cite[Chap.~6]{Fridman2014}.}
		\label{Fig:ErrorNormBR}
	\end{figure}
	
	\begin{figure}[htp!]
		\centering
		\includegraphics[width=0.47\textwidth]{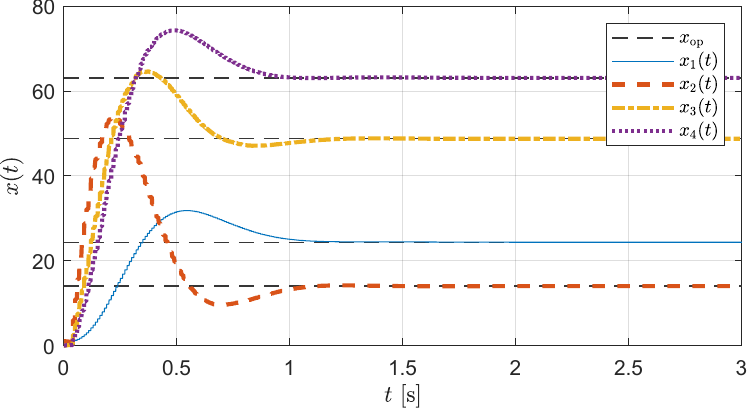}
		\caption{Response of the batch reactor in \eqref{Eq:ExmSys} to the stabilization process around the operational point $x_{\mathrm{op}}$ in \eqref{Eq:ExampOP} using $K$ in \eqref{Eq:GainsBR01}.}
		\label{Fig:TrajBR}
	\end{figure}	
	
	\begin{figure}[htp!]
		\centering
		\includegraphics[width=0.47\textwidth]{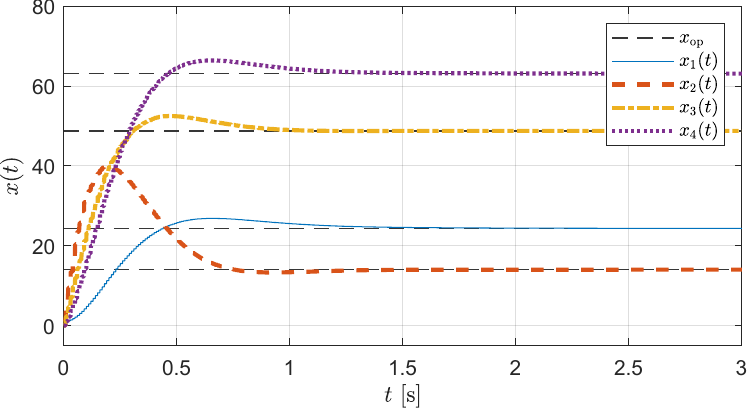}
		\caption{Response of the batch reactor in \eqref{Eq:ExmSys} to the stabilization process around the operational point $x_{\mathrm{op}}$ in \eqref{Eq:ExampOP} using $K$ in \eqref{Eq:GainsBR02} computed with noisy data.}
		\label{Fig:TrajBR2}
	\end{figure}

\section{Conclusions}	

\label{Sec:Conclusion}
	In this work we have presented a method for designing robust controllers for LTD-TDSs relying exclusively on input-state data recorded from the system, i.e., avoiding the system modeling. We have provided explicit data-dependent formulas to compute state feedback gains for stabilization, guaranteed cost control and $H_{\infty}$ control. By accounting on possible noise and unknown constant delays in the recorded data, the method ensures closed-loop stability of the system with the computed gain even under such circumstances.

	Differently from other methods based on data \cite{Persis2019,vanWaarden2020}, we have investigated robustness against uncertain delays. The proposed design approach provides stability guarantees on the closed-loop system through a robust control design. Furthermore, the amount of data required for the control design is relatively small as is shown in the numerical example.
	
	Future work will be geared towards the implementation and experimental validation of the reported results in real-world applications, such as traffic control or power systems operation. Likewise, we plan to investigate extensions to nonlinear systems, possibly by incorporating prior system knowledge as recently proposed in \cite{Berberich2020_2}.

\appendix
\renewcommand{\theequation}{A.\arabic{equation}}
	\begin{figure*}[tb!]
		\begin{align}
		\Phi_1=\left[\begin{array}{c c c c c}
		R_1+R_2 & -R_1+S_{12,1} & -R_2+S_{12,2} & -S_{12,1}-S_{12,2} & 0\\
		\star & 2R_1-S_{12,1}-S^\top_{12,1} & 0 & -R_1+S_{12,1} & 0\\
		\star & \star & 2R_2-S_{12,2}-S^\top_{12,2} & -R_2+S_{12,2} & 0\\
		\star & \star & \star & R_1+R_2 & 0\\
		\star & \star & \star & \star & 0
		\end{array}\right]. \label{App:Phi01}
		\end{align}
	\end{figure*}
	
	\begin{figure*}[t!]
		\begin{align}			
			\Phi_2=\left[\begin{array}{c c c c c}
			 \Phi_{2,11} & P^\top_2X_{\{1\}}G_K\begin{bmatrix}0\\ I_n\\ 0\end{bmatrix} & P^\top_2X_{\{1\}}G_K\begin{bmatrix}I_n\\ 0\\ 0\end{bmatrix} & 0 & -P^\top_2+\left(P^\top_3X_{\{1\}}G_K\begin{bmatrix}0\\0\\I_n\end{bmatrix}\right)^\top\\
			\star & 0 & 0 & 0 & \left(P^\top_3X_{\{1\}}G_K\begin{bmatrix}0\\ I_n\\ 0\end{bmatrix}\right)^\top\\
			\star & \star & 0 & 0 & \left( P^\top_3X_{\{1\}}G_K\begin{bmatrix}I_n\\ 0\\ 0\end{bmatrix}\right)^\top\\
			\star & \star & \star & 0 & 0\\
			\star & \star & \star & \star & -\left(P_3+P^\top_3\right)
			\end{array}\right],
			\label{App:Phi02}
		\end{align}
		\begin{gather*}
			\Phi_{2,11}=P^\top_2X_{\{1\}}G_K\begin{bmatrix}0\\0\\ I_n\end{bmatrix}+\left(P^\top_2X_{\{1\}}G_K\begin{bmatrix}0\\0\\ I_n\end{bmatrix}\right)^\top.
		\end{gather*}
	\end{figure*}
	
	\begin{figure*}[p]
		\begin{gather}
		\Phi = \left[\begin{array}{c c c c c}
		\Phi_{11} & \Phi_{12} & \Phi_{13} & -S_{12,1}-S_{12,2} & \Phi_{15}\\
		\star & 2R_1-S_{12,1}-S^\top_{12,1} & 0 & -R_1+S_{12,1} & -\left( P^\top_3X_{\{1\}}G_K\begin{bmatrix}0\\I_n\\0\end{bmatrix}\right)^\top\\
		\star & \star & 2R_2-S_{12,2}-S^\top_{12,2} & -R_2+S_{12,2} & -\left( P^\top_3X_{\{1\}}G_K\begin{bmatrix}I_n\\0\\0\end{bmatrix}\right)^\top\\
		\star & \star & \star & R_1+R_2+S & 0\\
		\star & \star & \star & \star & \Phi_{55}
		\end{array}\right], \label{App:Phi}
		\end{gather}
		\vspace*{-10pt}
		\begin{gather*}				
			\Phi_{11}=P-S+(1-\bar{h}^2)(R_1+R_2)-P^\top_2X_{\{1\}}G_K\begin{bmatrix}0\\0\\ I_n\end{bmatrix}-\left(P^\top_2X_{\{1\}}G_K\begin{bmatrix}0\\0\\ I_n\end{bmatrix}\right)^\top,\\ \Phi_{12}=-R_1+S_{12,1}-P^\top_2X_{\{1\}}G_K\begin{bmatrix}0\\ I_n\\ 0\end{bmatrix},\; \Phi_{13}=-R_2+S_{12,2}-P^\top_2X_{\{1\}}G_K\begin{bmatrix}I_n\\ 0\\ 0\end{bmatrix},\\ \Phi_{15}=\bar{h}^2(R_1+R_2)+P^\top_2-\left(P^\top_3X_{\{1\}}G_K\begin{bmatrix}0\\0\\I_n\end{bmatrix}\right)^\top,\;\Phi_{55}=-P-\bar{h}^2(R_1+R_2)+(P_3+P^\top_3).
		\end{gather*}
	\end{figure*}
	
	\begin{figure*}[p]
		\begin{gather}
		\bar{\Phi} = \left[\begin{array}{c c c c c}
		\bar{\Phi}_{11} & \bar{\Phi}_{12} & \bar{\Phi}_{13} & -\bar{S}_{12,1}-\bar{S}_{12,2} & \bar{\Phi}_{15}\\
		\star & 2\bar{R}_1-\bar{S}_{12,1}-\bar{S}^\top_{12,1} & 0 & -\bar{R}_1+\bar{S}_{12,1} & -\left(\varepsilon X_{\{1\}}G_K\begin{bmatrix}0\\I_n\\0\end{bmatrix}\bar{P}_2\right)^\top\\
		\star & \star & 2\bar{R}_2-\bar{S}_{12,2}-\bar{S}^\top_{12,2} & -\bar{R}_2+\bar{S}_{12,2} & -\left(\varepsilon X_{\{1\}}G_K\begin{bmatrix}I_n\\0\\0\end{bmatrix}\bar{P}_2\right)^\top\\
		\star & \star & \star & \bar{R}_1+\bar{R}_2+\bar{S} & 0\\
		\star & \star & \star & \star & -\bar{P}-\bar{h}^2(\bar{R}_1+\bar{R}_2)+\varepsilon (\bar{P}_2+\bar{P}^\top_2)
		\end{array}\right], \label{App:barPhi}
		\end{gather}
		\vspace*{-10pt}
		\begin{gather*}
			\bar{\Phi}_{11}=\bar{P}-\bar{S}+(1-\bar{h}^2)(\bar{R}_1+\bar{R}_2)-X_{\{1\}}G_K\begin{bmatrix}0\\0\\ I_n\end{bmatrix}\bar{P}_2-\left(X_{\{1\}}G_K\begin{bmatrix}0\\0\\ I_n\end{bmatrix}\bar{P}_2\right)^\top,\\
			\bar{\Phi}_{12}=-\bar{R}_1+\bar{S}_{12,1}-X_{\{1\}}G_K\begin{bmatrix}0\\ I_n\\ 0\end{bmatrix}\bar{P}_2,\;
			\bar{\Phi}_{13}=-\bar{R}_2+\bar{S}_{12,2}-X_{\{1\}}G_K\begin{bmatrix}I_n\\ 0\\ 0\end{bmatrix}\bar{P}_2,\\
			\bar{\Phi}_{15}=\bar{h}^2(\bar{R}_1+\bar{R}_2)+\bar{P}_2-\varepsilon\left(X_{\{1\}}G_K\begin{bmatrix}0\\0\\I_n\end{bmatrix}\bar{P}_2\right)^\top.
		\end{gather*}
	\end{figure*}
	
	\begin{figure*}[p]
		\begin{align}
			\tilde{\Phi}_2=\left[\begin{array}{c c c c c}
			 P^\top_2\Delta_{[\cdot]}G_K\begin{bmatrix}0\\0\\I_n\end{bmatrix}+\left(P^\top_2\Delta_{[\cdot]}G_K\begin{bmatrix}0\\0\\I_n\end{bmatrix}\right)^\top & P^\top_2\Delta_{[\cdot]} G_K\begin{bmatrix}0\\ I_n\\ 0\end{bmatrix} & P^\top_2\Delta_{[\cdot]} G_K\begin{bmatrix}I_n\\ 0\\ 0\end{bmatrix} & 0 & \left(\varepsilon P^\top_2\Delta_{[\cdot]} G_K\begin{bmatrix}0\\0\\I_n\end{bmatrix}\right)^\top\\
			\star & 0 & 0 & 0 & \left(\varepsilon P^\top_2\Delta_{[\cdot]} G_K\begin{bmatrix}0\\I_n\\0\end{bmatrix}\right)^\top\\
			\star & \star & 0 & 0 & \left(\varepsilon P^\top_2\Delta_{[\cdot]} G_K\begin{bmatrix}I_n\\0\\0\end{bmatrix}\right)^\top\\
			\star & \star & \star & 0 & 0\\
			\star & \star & \star & \star & 0
			\end{array}\right]. \label{App:TildePhi2}
		\end{align}
	\end{figure*}

	\begin{proof}[Proof of Proposition \ref{Prop:OpenLoopRep}]
		The matrices $A_0$, $A_1$ and $B$ of the system \eqref{Eq:System} are related through data by
		\begin{align}
			X_{\{1\}}&=\begin{bmatrix}B & A_1 & A_0\end{bmatrix}W_0.
			\label{AppEq:SMatrixData}
		\end{align} 
		\textit{Sufficiency:} Since by assumption $W_0$ has full-row rank, we obtain from \eqref{AppEq:SMatrixData} that

		\begin{align}
			\begin{bmatrix}B & A_1 & A_0\end{bmatrix}&= X_{\{1\}}W^{\dagger}_0. \label{AppEq:SysMatRep}
		\end{align}
		Using \eqref{AppEq:SysMatRep} in \eqref{Eq:System2} yields \eqref{Eq:OpenLoopRep}. 
		
		\textit{Necessity:} If $\mathrm{rank}(W_0)<m+2n$, then \eqref{AppEq:SMatrixData} together with the Rouché-Capelli theorem \cite{Shafarevich2012} implies that the matrices $A_0$, $A_1$ and $B$ of the system \eqref{Eq:System} cannot be determined univocally. 
		
		Therefore, the condition $\mathrm{rank}(W_0)=m+2n$ is necessary and sufficient to represent the open-loop system through data.
	\end{proof}

	\begin{proof}[Proof of Lemma \ref{Lem:ClosedLoopRep}]
		The closed-loop system \eqref{Eq:ClosedLoopSys} can be rewritten as
		\begin{align*}
			x(k+1)=\begin{bmatrix}B & A_1 & A_0\end{bmatrix}\begin{bmatrix}K & 0 & 0\\ 0 & I_n & 0\\ 0 & 0 & I_n\end{bmatrix}\begin{bmatrix}x_{h_2}(k) \\ x_{h_1}(k)\\ x(k)\end{bmatrix}.
		\end{align*}
		\textit{Sufficiency:} Since by assumption $\mathrm{rank}(W_0)=m+2n$, one has that
		\begin{align*}
			\mathrm{rank}\left(\left[\begin{array}{c | c}\begin{bmatrix}K & 0 & 0\\ 0 & I_n & 0\\ 0 & 0 & I_n\end{bmatrix} & W_0\end{array}\right]\right)=\mathrm{rank}(W_0).
		\end{align*}
		Thus, by the Rouché-Capelli theorem \cite{Shafarevich2012}, there exists a $T\times 3n$ matrix $G_K$, such that \eqref{Eq:GK} holds. Therefore,
		\begin{align*}
			\begin{bmatrix}B & A_1 & A_0\end{bmatrix}\begin{bmatrix}K & 0 & 0\\ 0 & I_n & 0\\ 0 & 0 & I_n\end{bmatrix}&=\begin{bmatrix}B & A_1 & A_0\end{bmatrix}W_0G_k\\
				&=X_{\{1\}}G_K,
		\end{align*}
		where the relation \eqref{AppEq:SMatrixData} has been used. 
		
		\textit{Necessity:} If $\mathrm{rank}(W_0)<m+2n$, then by the Rouché-Capelli theorem \cite{Shafarevich2012}, not for any matrix $\mathrm{diag}(K,I_n,I_n)$ there is a matrix $G_K$ that satisfies \eqref{Eq:GK}.
		
		This proves the main claim of Lemma \ref{Lem:ClosedLoopRep}. The explicit formula for $K$ in \eqref{Eq:Kformula} is obtained from \eqref{Eq:GK} by considering the definition of $W_0$ in \eqref{Eq:W0}.		
	\end{proof}
	
	\begin{proof}[Proof of Theorem \ref{Thm:Stabilization}]
		Inspired by the procedure of \cite[Sec.~6.1.3.2]{Fridman2014}, we propose the following Lyapunov-Krasovskii functional
		\begin{align}
			V(k)&=V_P(k)+V_S(k)+\sum^{2}_{i=1}V_{R,i}(k), \label{App:LKF}
		\end{align}
		with
		\begin{align*}
			V_P(k)&=x^\top(k)Px(k),\\
			V_S(k)&=\sum^{k-1}_{j=k-\bar{h}}x^\top(j)Sx(j),\\
			V_{R,i}(k)&=\bar{h}\sum^{-1}_{m=-\bar{h}} \sum^{k-1}_{j=k+m} \bar{y}^\top(j)R_i\bar{y}(j),\\
			\bar{y}(j):&=x(j+1)-x(j),
		\end{align*}
		where $P>0$, $S>0$, $R_i>0$, with $i=\{1,2\}$, are $n\times n$ matrix variables. We are interested in computing $V(k+1)-V(k)$, which can be done by taking the difference of each of its components. By direct calculation, the following difference equations are obtained:
		\begin{align*}
			V_{P}(k+1)-V_P(k)&=x^\top(k+1)Px(k+1)-x^\top(k)Px(k),\\
			V_S(k+1)-V_S(k)&=x^\top(k)Sx(k)-x^\top_{\bar{h}}(k)Sx_{\bar{h}}(k),\\
			V_{R,i}(k+1)-V_{R,i}(k)&=\bar{h}^2\bar{y}^\top(k)R_i\bar{y}(k)\\
				&\qquad\qquad -\bar{h}\sum^{k-1}_{j=k-\bar{h}}\bar{y}^\top(j)R_i\bar{y}(j).
		\end{align*}
		The summation term in $V_{R,i}(k+1)-V_{R,i}(k)$ is rewritten as
		\begin{align*}
			\bar{h}\sum^{k-1}_{j=k-\bar{h}}\bar{y}^\top(j)R_i\bar{y}(j)&=\\
				\bar{h}\sum^{k-1-h_{i}(k)}_{j=k-\bar{h}}&\bar{y}^\top(j) R_i\bar{y}(j)+\bar{h}\sum^{k-1}_{j=k-h_{i}(k)}\bar{y}^\top(j)R_i\bar{y}(j).
		\end{align*}
		By applying Jensen's inequality twice \cite[Sec.~6.1.3.2]{Fridman2014}, one obtains
		\begin{align*}
			\bar{h}\sum^{k-1-h_{i}(k)}_{j=k-\bar{h}}&\bar{y}^\top(j)R_i\bar{y}(j)\geq\\
				\frac{\bar{h}}{\bar{h}-h_{i}(k)}&\left[x_{h_{i}(k)}(k)-x_{\bar{h}}(k)\right]^\top R_i \left[x_{h_{i}(k)}(k)-x_{\bar{h}}(k)\right],
		\end{align*}
		\begin{align*}
			\bar{h}\sum^{k-1}_{j=k-h_{i}(k)}&\bar{y}^\top(j)R_i\bar{y}(j)\geq\\
				\frac{\bar{h}}{h_{i}(k)}&\left[x(k)-x_{h_{i}(k)}(k)\right]^\top R_i\left[x(k)-x_{h_{i}(k)}(k)\right].
		\end{align*}
		Furthermore, by invoking the reciprocally convex approach \cite[Lem.~3.4]{Fridman2014} one has
		\begin{align}
			&\bar{h}\sum^{k-1}_{j=k-\bar{h}}\bar{y}^\top(j)R_i\bar{y}(j)\geq\nonumber\\
				&\;\frac{\bar{h}}{\bar{h}-h_{i}(k)}\left[x_{h_{i}(k)}(k)-x_{\bar{h}}(k)\right]^\top R_i \left[x_{h_{i}(k)}(k)-x_{\bar{h}}(k)\right]\nonumber\\
				&\;+\frac{\bar{h}}{h_{i}(k)}\left[x(k)-x_{h_{i}(k)}(k)\right]^\top R_i\left[x(k)-x_{h_{i}(k)}(k)\right]\geq \nonumber\\
			&\begin{bmatrix}x(k)-x_{h_{i}(k)}(k)\\ x_{h_{i}(k)}(k)-x_{\bar{h}}(k)\end{bmatrix}^\top \begin{bmatrix}R_i & S_{12,i}\\ \star & R_i\end{bmatrix}\begin{bmatrix}x(k)-x_{h_{i}(k)}(k)\\ x_{h_{i}(k)}(k)-x_{\bar{h}}(k)\end{bmatrix},\label{App:RecCon01}
		\end{align}
		for any $S_{12,i}\in\mathbb{R}^{n\times n}$ satisfying
		\begin{align*}
			\begin{bmatrix}R_i & S_{12,i}\\ \star & R_i\end{bmatrix}\geq 0.
		\end{align*}
		Consider the short-hand
		\begin{align}
			\chi^\top(k)=\left[x^\top(k), x^\top_{h_{1}(k)}(k), x^\top_{h_{2}(k)}(k), x^\top_{\bar{h}}(k), x^\top(k+1)\right]. \label{App:shorthand01}
		\end{align}
		Then, by using \eqref{App:RecCon01} and \eqref{App:shorthand01} we obtain 
		\begin{align}
			\sum_{i=1}^{2}V_{R,i}(k+1)-V_{R,i}(k)&\leq \bar{h}^2\bar{y}^\top(k)\big(R_1+R_2\big)\bar{y}(k)\nonumber\\
				&\qquad\qquad -\chi^\top(k)\Phi_1\chi(k),
			\label{App:IneMatrix01}
		\end{align}
		where $\Phi_1$ is given in \eqref{App:Phi01}. 
		
		Now, consider \eqref{Eq:DataRepCL} and \eqref{Eq:GK} in Lemma \ref{Lem:ClosedLoopRep}. From them it follows that
		\begin{align}
			\begin{split}
				BK &= X_{\{1\}}G_K\begin{bmatrix}I_n & 0 & 0\end{bmatrix}^\top,\\
				A_1 &= X_{\{1\}}G_K\begin{bmatrix}0 & I_n & 0\end{bmatrix}^\top,\\
				A_0 &= X_{\{1\}}G_K\begin{bmatrix}0 & 0 & I_n\end{bmatrix}^\top,
			\end{split}	\label{Eq:DBMatrices}
		\end{align}		
		subject to
		\begin{align}
			\begin{split}
				0 &= U_{h_2,\{0\}}G_K\begin{bmatrix}0 & I_n & 0\end{bmatrix}^\top=U_{h_2,\{0\}}G_K\begin{bmatrix}0 & 0 & I_n\end{bmatrix}^\top,\\
				0 &= X_{h_1,\{0\}}G_K\begin{bmatrix}I_n & 0 & 0\end{bmatrix}^\top = X_{h_1,\{0\}}G_K\begin{bmatrix}0 & 0 & I_n\end{bmatrix}^\top,\\
				0 &= X_{\{0\}}G_K\begin{bmatrix}I_n & 0 & 0\end{bmatrix}^\top = X_{\{0\}}G_K\begin{bmatrix}0 & I_n & 0\end{bmatrix}^\top.
			\end{split} \label{Eq:DBRestrictions}
		\end{align}
				
		Now, combining the descriptor method \cite[Sec.~3.5.2]{Fridman2014} with the data-based representation of the system matrices in \eqref{Eq:DBMatrices}, we obtain
		\begin{align}
			0&=2\left[x^\top(k)P^\top_2+x^\top(k+1)P^\top_3\right]\times \nonumber\\
				&\quad\times\left[X_{\{1\}}G_K\left[x^\top_{h_{2}(k)}(k),x^\top_{h_{1}(k)}(t),x^\top(k)\right]^\top-x(k+1)\right] \nonumber\\
			&=-2x^\top(k)P^\top_2x(k+1)-2 x^\top(k+1)P^\top_3x(k+1) \nonumber\\
			&\qquad +2\big(x^\top (k)P^\top_2+x^\top(k+1)P^\top_3\big)X_{\{1\}}G_K\times \nonumber\\
			&\qquad\times \left(\begin{bmatrix}I_n\\0\\0\end{bmatrix}x_{h_{2}(k)}(k)+\begin{bmatrix}0\\I_n\\0\end{bmatrix}x_{h_{1}(k)}(k)+\begin{bmatrix}0\\0\\I_n\end{bmatrix}x(k)\right), \label{App:IneMatrix02}
		\end{align}
		where $P_2\in\mathbb{R}^{n\times n}$ and $P_3\in\mathbb{R}^{n\times n}$ are matrix variables. By using the short-hand \eqref{App:shorthand01}, we can rewrite \eqref{App:IneMatrix02} as the quadratic form 
		\begin{align}
			0=\chi^\top(k)\Phi_2\chi(k),
			\label{App:IneMatrix03}
		\end{align}
		with $\Phi_2$ given in \eqref{App:Phi02}. 
		
		Now, retaking the calculation $V(k+1)-V(k)$, by considering \eqref{App:IneMatrix01}, \eqref{App:IneMatrix03}, and adding $\Phi_1$ and $\Phi_2$ we obtain
		\begin{align}
			V(k+1)-V(k)\leq -\chi^\top(k)\Phi\,\chi(k), \label{App:DiffV}
		\end{align}
		where $\Phi$ is given in \eqref{App:Phi}. 
		
		By recalling that both $G_K$ and $K$ are design parameters, an inspection of $\Phi$ in \eqref{App:Phi} reveals that it contains nonlinear terms in the decision variables. In order to reformulate $\Phi$ as an LMI, we follow the standard approach in control design of TDSs via the descriptor method, see \cite[Chapter 6]{Fridman2014}. Since $\Phi>0$ in \eqref{App:Phi} implies that $\Phi_{55}>0$, we have $P_3+P^\top_3>0$, meaning that $P_3$ is invertible. We choose $P_3=\varepsilon P_2,$ where $\varepsilon>0$ is a scalar tuning parameter. Then, we define $\bar{P}_2=P^{-1}_2$ and the matrices
		\begin{align}
			\begin{split}
				&\bar{P}=\bar{P}^\top_2P\bar{P}_2, \qquad \bar{R}_i=\bar{P}^\top_2R_i\bar{P}_2,\\ 
				&\bar{S}=\bar{P}^\top_2S\bar{P}_2, \qquad \bar{S}_{12,i}=\bar{P}^\top_2S_{12,i}\bar{P}_2.
			\end{split} \label{App:CongMat}
		\end{align}
		Consider the congruent transformation $\bar{\Phi}=\mathbf{P}^\top \Phi\mathbf{P}$, with
		\begin{align*}
			\mathbf{P}=\mathrm{diag}\big(\bar{P}_2,\,\bar{P}_2,\,\bar{P}_2,\,\bar{P}_2,\,\bar{P}_2\big).
		\end{align*}
		The resulting matrix $\bar{\Phi}$ is given in \eqref{App:barPhi}, which is linear in all decision variables. Now, inspired by \cite{Persis2019}, we introduce the auxiliary matrix variables
		\begin{align}
			Q_1=G_K\begin{bmatrix}I_n\\0\\0\end{bmatrix}\bar{P}_2,\,
			Q_2=G_K\begin{bmatrix}0\\I_n\\0\end{bmatrix}\bar{P}_2,\,
			Q_3=G_K\begin{bmatrix}0\\0\\I_n\end{bmatrix}\bar{P}_2. \label{App:DataSubs}
		\end{align}
		By considering \eqref{Eq:W0} and the restriction \eqref{Eq:GK}, it holds that
		\begin{align}
			\begin{split}
				K\bar{P}_2&=U_{h_2,\{0\}}G_K\begin{bmatrix}I_n\\0\\0\end{bmatrix}\bar{P}_2=U_{h_2,\{0\}}Q_1,\\
				\bar{P}_2&=X_{h_1,\{0\}}G_K	\begin{bmatrix}0 \\ I_n\\ 0\end{bmatrix}\bar{P}_2 =	X_{h_1,\{0\}}Q_2,\\
				\bar{P}_2&=X_{\{0\}}G_K\begin{bmatrix}0\\0\\I_n\end{bmatrix}\bar{P}_2=X_{\{0\}}Q_3.
			\end{split} \label{App:ControlEq01}
		\end{align}	
		Replacing the previous definitions in $\bar{\Phi}$ in \eqref{App:barPhi} yields the matrix $\bar{\mathbf{\Phi}}$ in \eqref{Eq:barBFPhi} and the data-based set of inequalities \eqref{Eq:Thm01a} and \eqref{Eq:Thm01b}. The equality restrictions \eqref{Eq:QRestrictions} follow from \eqref{Eq:DBRestrictions} and the definitions \eqref{App:DataSubs} and \eqref{App:ControlEq01}. The control gain $K$ is obtained from \eqref{App:ControlEq01} and is given in \eqref{Eq:Thm01control}. Then, if the above-mentioned inequalities hold, we have $V(k+1)-V(k)<0$. By invoking \cite[Thm.~6.1]{Fridman2014}, the assertions of the theorem follow.
	\end{proof}	
		
	\begin{proof}[Proof of Corollary \ref{Cor:GuaranteedCostControl}]
		Consider the Lyapunov-Krasovskii functional given in \eqref{App:LKF} and the short-hand introduced in \eqref{App:shorthand01}. From \eqref{App:DiffV} together with the performance output $z(k)$ given in \eqref{Eq:PerfOutput}, we have
		\begin{align*}
			V(k+1)-V(k)+z^\top(k)z(k)\leq -\hat{\chi}^\top(k)\Psi\hat{\chi}(k),
		\end{align*}
		where $\hat{\chi}^\top(k)=[\chi^\top(k),z^\top(k)]$ and
		\begin{gather}
			\Psi = \left[\begin{array}{c c}\Phi & -\kappa^\top\\ \star & I_q\end{array}\right],\label{App:Psi}\\
			\kappa= \begin{bmatrix}L_1 & L_2 & DK & 0 & 0\end{bmatrix}.\nonumber
		\end{gather}
		The matrix $\Phi$ is given in \eqref{App:Phi}. Consider once more the matrices introduced in \eqref{App:CongMat}. As before, in order to obtain an LMI, we perform a congruent transformation. Define the block-diagonal matrix $\mathbf{P}$
		\begin{align*}
			\mathbf{P}=\mathrm{diag}\big(\bar{P}_2,\bar{P}_2,\bar{P}_2,\bar{P}_2,\bar{P}_2,I_q\big).
		\end{align*}
		Using $\mathbf{P}$, we define the congruent matrix $\bar{\Psi}=\mathbf{P}^\top\Psi\mathbf{P}$, and, by considering the substitutions \eqref{App:DataSubs}, we obtain the matrix
		\begin{gather*}
			\bar{\Psi} = \left[\begin{array}{c c}\bar{\Phi} & -\bar{\kappa}^\top\\ \star & I_q\end{array}\right],\\
			\bar{\kappa}= \begin{bmatrix}L_1\bar{P}_2 & L_2\bar{P}_2 & DK\bar{P}_2 & 0 & 0\end{bmatrix},\nonumber
		\end{gather*}
		with $\bar{\Phi}$ given in \eqref{App:barPhi}. By replacing $\bar{P}_2$ following \eqref{App:ControlEq01}, it results $\bar{\mathbf{\Psi}}$ in \eqref{Eq:barBFPsi}. By invoking \cite[Prop.~6.5]{Fridman2014}, we have that $J\leq V(x_0)=x^\top(0)Px(0)$ if the data-based inequalities \eqref{Eq:barBFPsi} and \eqref{Eq:Cor01b} are satisfied. In addition, we have $V(0)\leq \delta$, and by transitivity $J\leq \delta$, if \eqref{Eq:Cor01c} is simultaneously satisfied. This is achieved by selecting the feedback gain $K$ according to \eqref{Eq:Cor01control}. 
	\end{proof}

	\begin{proof}[Proof of Corollary \ref{Cor:HinftyControl}]
		Consider the performance index 
		\begin{align*}
			J_{\infty}=\sum^{\infty}_{k=0}\left(z^\top(k)z(k)-\gamma\omega^\top(k)\omega(k)\right).
		\end{align*}
		If one can find a feedback gain $K$ such that $J_{\infty}<0$, then the system \eqref{Eq:PerturbedSys} has an $L_2$-gain less than $\gamma$ \cite{Fridman2014}. Such gain can be found as follows. Consider the short hand $$\hat{\chi}^\top_2(k)=[\hat{\chi}^\top(k),\omega^\top(k)].$$ From \eqref{App:DiffV} together with the system dynamics \eqref{Eq:PerturbedSys} and \eqref{Eq:PerfOutputH}, we obtain
		\begin{align}
			\begin{split}
				-&\hat{\chi}^\top_2(k)\Gamma\hat{\chi}_2(k)\geq\\
				&\quad V(k+1)-V(k)+\left(z^\top(k)z(k)-\gamma\omega^\top(k)\omega(k)\right),
			\end{split}
		\end{align}
		with 
		\begin{gather}
			\Gamma = \left[\begin{array}{c c}\Psi & \kappa^\top_2\\ \star & \gamma\,I_p\end{array}\right],\\
			\kappa_2=\begin{bmatrix}-D^\top_0P_2 & 0 & 0 & 0 & -\varepsilon D^\top_0P_2 & 0\end{bmatrix},\nonumber
		\end{gather}
		where $\Psi$ is given in \eqref{App:Psi}. As in the proof of Theorem \ref{Thm:Stabilization} and Corollary \ref{Cor:GuaranteedCostControl}, we consider the congruent matrix $\bar{\Gamma}=\mathbf{P}^\top\Gamma\mathbf{P}$ with
		\begin{align*}
			\mathbf{P}=\mathrm{diag}\big(\bar{P}_2,\bar{P}_2,\bar{P}_2,\bar{P}_2,\bar{P}_2,I_q,I_p\big),
		\end{align*}
		where $\bar{P}_2=P^{-1}_2$. This is done to obtain an LMI from the original BMI. By taking into account \eqref{App:CongMat}, \eqref{App:DataSubs} and \eqref{App:ControlEq01}, we obtain the matrix $\bar{\mathbf{\Gamma}}$ in \eqref{Eq:barBFGamma}. Hence, if \eqref{Eq:barBFGamma} and \eqref{Eq:Cor02b} are satisfied, the gain given in \eqref{Eq:Cor02control} guarantees a $L_2$ gain of $\gamma$ for $\omega(k)\neq 0$.
	\end{proof}

	\begin{proof}[Proof of Proposition \ref{Prop:DelayIdent}]
		Let $h_1=j^\star$ and $h_2=i^\star$ hold, and consider the short-hands $W_{0\star}=W_{0,(i^\star,j^\star)}$, $W^{\mathrm{nom}}_{0\star}=W^{\mathrm{nom}}_{0,(i^\star,j^\star)}$ and $W^{\delta}_{0\star}=W^{\delta}_{0,(i^\star,j^\star)}$. In the unperturbed case (i.e., $W^{\delta}_{0\star}=0$ and $\,X^{\delta}_{\{1\}}=0$), the orthogonal distance $d_{(i^\star,j^\star)}(X_{\{1\}})$ defined in \eqref{Eq:OrthDis} is zero, in other words \eqref{Eq:DelayIdentNom} holds. In contrast, in the perturbed case ($W^{\delta}_{0\star}\neq 0,\,X^{\delta}_{\{1\}}\neq 0$) we have
		\begin{align}
			d_{(i^\star,j^\star)}&\big(X_{\{1\}}\big)=\left\|X_{\{1\}}\left(I_T-W_{0\star}^{\dagger}W_{0\star}\right)\right\|_2 \nonumber\\
				&= \bigg\|\Big(X^{\mathrm{nom}}_{\{1\}}+X^{\delta}_{\{1\}}\Big)\left(I_T-W_{0\star}^{\dagger}W_{0\star}\right)\bigg\|_2. \label{App:Distance01}
		\end{align}
		Therefore, it follows that
		\begin{align}
			d_{(i^\star,j^\star)}\big(X_{\{1\}}\big)\leq d_{(i^\star,j^\star)}\big(X^{\mathrm{nom}}_{\{1\}}\big)+d_{(i^\star,j^\star)}\big(X^{\delta}_{\{1\}}\big). \label{App:DistAux}
		\end{align}
		From the properties of orthogonal projectors, we immediately have 
		\begin{align}
			d_{(i^\star,j^\star)}(X^{\delta}_{\{1\}})\leq \|X^{\delta}_{\{1\}}\|_2\leq r_{X^{\delta}_{\{1\}}}. \label{App:DistAux0}
		\end{align}
		Now, we proceed to upper bound $d_{(i^\star,j^\star)}(X^{\mathrm{nom}}_{\{1\}})$, for which we recall \eqref{Eq:PerturbedW0} and consider
		\begin{align}
			I_T-W_{0\star}^{\dagger}&W_{0\star} = I_T-\big(W^{\mathrm{nom}}_{0\star}\big)^{\dagger}W^{\mathrm{nom}}_{0\star}\nonumber\\
				&+\big(W^{\mathrm{nom}}_{0\star}\big)^{\dagger}W^{\mathrm{nom}}_{0\star}- W_{0\star}^{\dagger}W_{0\star} \label{App:DistAux01}.
		\end{align}		
		In account of \eqref{App:DistAux01} and the fact that
		\begin{align}
			X^{\mathrm{nom}}_{\{1\}}\left(I_T-\left(W^{\mathrm{nom}}_{0\star}\right)^{\dagger}W^{\mathrm{nom}}_{0\star}\right)=0, \label{App:DistAux01b}
		\end{align}
		we have from \eqref{App:Distance01} that
		\begin{align}
			d_{(i^\star,j^\star)}\big(X^{\mathrm{nom}}_{\{1\}}\big)=
			\bigg\|X^{\mathrm{nom}}_{\{1\}}\bigg(\big(W^{\mathrm{nom}}_{0\star}\big)^{\dagger}W^{\mathrm{nom}}_{0\star}- W_{0\star}^{\dagger}W_{0\star}\bigg) \bigg\|_2. \label{App:DistAux02}
		\end{align}
		Therefore, to obtain an upper bound for $d_{(i^\star,j^\star)}(X^{\mathrm{nom}}_{\{1\}})$ we need to bound the difference
		\begin{align*}
			\left\|\big(W^{\mathrm{nom}}_{0\star}\big)^{\dagger}W^{\mathrm{nom}}_{0\star}- W_{0\star}^{\dagger}W_{0\star}\right\|_2.
		\end{align*}
		Following \cite[Thm.~4.1]{Stewart1977}, and under Assumption \ref{Assum:Noise}, it holds that
		\begin{align*}
			\left\|\big(W^{\mathrm{nom}}_{0\star}\big)^{\dagger}W^{\mathrm{nom}}_{0\star}-W_{0\star}^{\dagger}W_{0\star}\right\|_2&\leq\\ \psi&\left(\frac{\left\|W^{\delta,12}_{0\star}\right\|_2\left\|W^{\dagger}_{0\star}\right\|_2}{1-\left\|W^{\delta,11}_{0\star}\right\|_2\left\|W^{\dagger}_{0\star}\right\|_2}\right),
		\end{align*}
		with $\psi(\cdot)$ defined in \eqref{App:DistAuxpsi}, and $\|W^{\delta,11}_{0\star}\|_2$ and $\|W^{\delta,12}_{0\star}\|_2$ as in \eqref{Eq:W0decomp}. Given that $\psi$ is a monotonically increasing function of its argument and that with Assumption \ref{Assum:Noise} the relation \eqref{Eq:PropDelay01} holds, we have that
		\begin{align*}
			1-\left\|W^{\delta,11}_{0\star}\right\|_2\left\|W^{\dagger}_{0\star}\right\|_2\geq 1-r_{W^{\delta,11}_{0,(i,j)}}\left\|W^{\dagger}_{0\star}\right\|_2>0,
		\end{align*}	
		it follows that
		\begin{align}
			\left\|\big(W^{\mathrm{nom}}_{0\star}\big)^{\dagger}W^{\mathrm{nom}}_{0\star}-W_{0\star}^{\dagger}W_{0\star}\right\|_2&\leq \psi\left(\frac{r_{W^{\delta,12}_{0\star}}\left\|W^{\dagger}_{0\star}\right\|_2}{1-r_{W^{\delta,11}_{0\star}}\left\|W^{\dagger}_{0\star}\right\|_2}\right)\nonumber \\
			&=\psi(\sigma^\star), \label{App:DistAux03}
		\end{align}	 
		with $r_{W^{\delta,11}_{0\star}}$ and $r_{W^{\delta,12}_{0\star}}$ as in \eqref{Eq:PropDelay01} and $\sigma^\star$ given in \eqref{sigmastar}. Hence, from \eqref{App:DistAux02} and \eqref{App:DistAux03}, we obtain
		\begin{align}
					d_{i^\star,j^\star}\big(X^{\mathrm{nom}}_{\{1\}}\big)&\leq \big\|X^{\mathrm{nom}}_{\{1\}}\big\|_2\cdot \psi\left(\sigma^\star\right)\nonumber \\
						&\leq \big(\|X_{\{1\}}\|_2+r_{X^{\delta}_{\{1\}}}\big)\cdot \psi\left(\sigma^\star\right)\label{App:DistAux04}.
		\end{align}		
		In account of \eqref{App:DistAux} and the bounds \eqref{App:DistAux0} and \eqref{App:DistAux04}, the upper bound for $d_{(i^\star,j^\star)}(X_{\{1\}})$ given in \eqref{Eq:DelayCond} follows.

		For $i\neq i^\star$ and $j\neq j^\star$, \eqref{App:DistAux01b} does not hold, and thus, the upper bound for $d_{(i,j)}(X_{\{1\}})$ given in \eqref{Eq:DelayCond} increases.
	
		Once the correct delays are determined, the corresponding open- and closed-loop data-based system representations are obtained via Proposition~\ref{Prop:OpenLoopRep} and Lemma~\ref{Lem:ClosedLoopRep}, respectively.
		
		Finally, if for two distinct pairs $(i,j)$ the condition \eqref{Eq:DelayCond} holds simultaneously, there are two candidates for the delay values and it is not possible to distinguish between them with the derived bound \eqref{Eq:DelayCond} and with the available data.
	\end{proof}

	\begin{proof}[Proof of Theorem \ref{Thm:GainNoise}]
		As in Theorem~\ref{Thm:GainNoise}, let there exist the matrices $Q_i$ for $i=\{1,2,3\}$, $\bar{P}$, $\bar{S}$, $\bar{R}_i$ and $\bar{S}_{12,i}$ for $i=\{1,2\}$, and compute $K$ following \eqref{Eq:NoiseThm01control}. Following \eqref{App:ControlEq01}, define $\bar{P}_2=X_{\{0\}}Q_3$. By using \eqref{App:CongMat}, the matrices $P$, $S$, $R_i$ and $S_{12,i}$ for $i=\{1,2\}$ can be computed. With these matrices, the Lyapunov-Krasovskii functional \eqref{App:LKF} can be built and used to analyze the stability of \eqref{Eq:System} with feedback gain $K$. 
		
		Consider the proof of Theorem \ref{Thm:Stabilization}. The effect of the noise impacts the terms introduced by the descriptor method, i.e. \eqref{App:IneMatrix02}, since
		\begin{align}
			X_{\{1\}}G_K\begin{bmatrix}x_{h_{2}(k)}(k)\\ x_{h_{1}(k)}(k)\\ x(k)\end{bmatrix}-x(k+1)=&\nonumber\\
				\left(X_{\{1\}}G_K-\begin{bmatrix}BK & A_1 & A_0\end{bmatrix}\right)&\begin{bmatrix}x_{h_{2}(k)}(k)\\ x_{h_{1}(k)}(k)\\ x(k)\end{bmatrix}\neq 0. \label{App:Mismatch}
		\end{align}
		To account for this mismatch, we compute the error induced by the corrupted data. Consider the nominal part of the data $X^{\mathrm{nom}}_{\{1\}}$ and $W^{\mathrm{nom}}_0$. From Assumption \ref{Assum:Noise}.1 we have $\mathrm{rank}(W^{\mathrm{nom}}_0)=m+2n$. It follows that
		\begin{align*}
			\begin{bmatrix}B & A_1 & A_0\end{bmatrix} = X^{\mathrm{nom}}_{\{1\}} \left(W^{\mathrm{nom}}_0\right)^{\dagger}.
		\end{align*}
		From \eqref{Eq:GK} in combination with the expression above, we get
		\begin{align*}
			\begin{bmatrix}BK & A_1 & A_0\end{bmatrix} = X^{\mathrm{nom}}_{\{1\}} \left(W^{\mathrm{nom}}_0\right)^{\dagger}W_{0}G_K.
		\end{align*}
		By using this relation, we obtain
		\begin{align}
			X_{\{1\}}G_K-&\begin{bmatrix}BK & A_1 & A_0\end{bmatrix}= \nonumber\\
				&\left(X_{\{1\}}-X^{\mathrm{nom}}_{\{1\}} \left(W^{\mathrm{nom}}_0\right)^{\dagger}W_{0}\right)G_K. \label{App:NoiseMatrix}
		\end{align}
		Note that $X_{\{1\}}=X^{\mathrm{nom}}_{\{1\}}+X^{\delta}_{\{1\}}$ and $W_{0}=W^{\mathrm{nom}}_0+W^{\delta}_0$.  We further continue from \eqref{App:NoiseMatrix} as
		\begin{align*}
			X_{\{1\}}&-X^{\mathrm{nom}}_{\{1\}} \left(W^{\mathrm{nom}}_0\right)^{\dagger}W_{0}=\\
				&X^{\mathrm{nom}}_{\{1\}}\left(I_T-\left(W^{\mathrm{nom}}_0\right)^{\dagger}W_{0}\right)+X^{\delta}_{\{1\}} \\
				=&\, X^{\mathrm{nom}}_{\{1\}}\left(I_T-\left(W^{\mathrm{nom}}_0\right)^{\dagger}W^{\mathrm{nom}}_{0}-\left(W^{\mathrm{nom}}_0\right)^{\dagger}W^{\delta}_{0}\right)+X^{\delta}_{\{1\}}\\
				=&\,-X^{\mathrm{nom}}_{\{1\}}\left(W^{\mathrm{nom}}_0\right)^{\dagger}W^{\delta}_{0}+X^{\delta}_{\{1\}} \\
				=&\,-\begin{bmatrix}B & A_1 & A_0\end{bmatrix}W^{\delta}_{0}+X^{\delta}_{\{1\}}=-\Delta_{[\cdot]}. 
		\end{align*}
		Then, in order to account for the mismatch \eqref{App:Mismatch} and to keep the equation \eqref{App:IneMatrix02} equal to zero, we add the compensating term
		\begin{align*}
			&2\left[x^\top(k)+\varepsilon x^\top(k+1)\right]P^\top_2\times \Delta_{[\cdot]} G_K\begin{bmatrix}x_{h_{2k}}(k)\\ x_{h_{1k}}(k)\\ x(k)\end{bmatrix}
		\end{align*}
		to \eqref{App:IneMatrix02}. This term can be written as the quadratic form $\chi^\top(k)\tilde{\Phi}_2\chi(k)$ with $\tilde{\Phi}_2$ given in \eqref{App:TildePhi2}. Carrying this term along and continuing as in the proof of Theorem \ref{Thm:Stabilization}, we obtain
		\begin{align}
			V(k+1)-V(k)\leq -\chi^\top(k)\left(\bar{\Phi}-\tilde{\Phi}_2\right)\chi(k). \label{App:diffVNoise}
		\end{align}
		As before, to obtain an LMI from \eqref{App:diffVNoise}, we make use of a congruent transformation. Let $\mathbf{P}=\mathrm{diag}(\bar{P}_2,\bar{P}_2,\bar{P}_2,\bar{P}_2,\bar{P}_2)$, and consider the congruent transformation
		\begin{align*}
			\mathbf{P}^\top\left(\bar{\Phi}-\tilde{\Phi}_2\right)\mathbf{P}=\bar{\mathbf{\Phi}}-\tilde{\mathbf{\Phi}}_2,
		\end{align*}
		where $\bar{\mathbf{\Phi}}$ is given in \eqref{Eq:barBFPhi} and 
		\begin{align}
			\tilde{\mathbf{\Phi}}_2=\mathbf{\Delta}\mathbf{Q}+\mathbf{Q}^\top\mathbf{\Delta}^\top, \label{App:barPhi2}
		\end{align}
		with $\mathbf{\Delta}=\mathrm{diag}\left(\Delta_{[\cdot]},\Delta_{[\cdot]},\Delta_{[\cdot]},\Delta_{[\cdot]},\Delta_{[\cdot]}\right)$ and $\mathbf{Q}$ is given in \eqref{Eq:bfQ}. Using \eqref{App:barPhi2}, it follows that
		\begin{align*}
			\tilde{\mathbf{\Phi}}_2 \leq \lambda \mathbf{\Delta}\mathbf{\Delta}^\top+\frac{1}{\lambda}\mathbf{Q}^\top\mathbf{Q} \leq \alpha^2 \lambda I_{5n} +\frac{1}{\lambda}\mathbf{Q}^\top\mathbf{Q},
		\end{align*}
		with $\lambda>0$ and $\alpha>0$ as in Theorem \eqref{Thm:GainNoise}. Therefore, the negativeness of \eqref{App:diffVNoise} can be ensured if
		\begin{align*}
			\bar{\mathbf{\Phi}}-\alpha^2 \lambda I_{5n} -\frac{1}{\lambda}\mathbf{Q}^\top\mathbf{Q}>0.
		\end{align*}
		By using the Schur complement, the inequality above is transformed into \eqref{Eq:NoiseThm01a}. Hence, if \eqref{Eq:NoiseThm01a} and \eqref{Eq:NoiseThm01b} are satisfied, it is ensured that $V(k+1)-V(k)<0$ and the origin of \eqref{Eq:System} with feedback $u(k)=Kx(k),$ where $K$ is given in \eqref{Eq:NoiseThm01control} is exponentially stable.
		
	\end{proof}

\bibliographystyle{IEEEtran}
\bibliography{IEEEabrv,ref_2}

\end{document}